\documentclass[11pt]{article}

\usepackage{latexsym}
\usepackage{amssymb}
\usepackage{amsthm}
\usepackage{amscd}
\usepackage{amsmath}
\usepackage{mathrsfs}
\usepackage{graphicx}
\usepackage{hyperref}
\usepackage{tabmac}
\usepackage{shuffle}

\usepackage[all]{xy}
\input xy \xyoption{frame}
\xyoption{dvips}

\usepackage[colorinlistoftodos]{todonotes}
\usepackage{fullpage}
\usepackage{ytableau}

\theoremstyle{definition}
\newtheorem* {theorem*}{Theorem}
\newtheorem* {conjecture*}{Conjecture}
\newtheorem{theorem}{Theorem}[section]
\newtheorem{thmdef}[theorem]{Theorem-Definition}
\newtheorem{propdef}[theorem]{Proposition-Definition}

\theoremstyle{definition}

\newtheorem* {example*}{Example}

\newtheorem{lemma}[theorem]{Lemma}
\theoremstyle{definition}
\newtheorem{definition}[theorem]{Definition}
\theoremstyle{definition}

\newtheorem{conjecture}[theorem]{Conjecture}
\newtheorem{proposition}[theorem]{Proposition}
\newtheorem{corollary}[theorem]{Corollary}
\newtheorem{algorithm}[theorem]{Algorithm}
\newtheorem{remark}[theorem]{Remark}
\theoremstyle{definition}
\newtheorem {example}[theorem]{Example}
\theoremstyle{definition}

\theoremstyle{definition}

\theoremstyle{definition}

\xyoption{dvips}

\def\({\left(}
\def\){\right)}

\newcommand{\CC}{\mathbb{C}}
\newcommand{\QQ}{\mathbb{Q}}
\newcommand{\cP}{\mathcal{P}}

\newcommand{\cR}{\mathcal{R}}

\def\cX{\mathcal{X}}
\def\cY{\mathcal{Y}}
\def\NN{\mathbb{N}}

\def\CC{\mathbb{C}}

\def\ZZ{\mathbb{Z}}

\def\GL{\mathrm{GL}}

\def\spanning{\textnormal{-span}}

\def\SMT{\mathrm{SMT}}
\def\SSMT{\mathrm{SSMT}}
\def\SSYT{\mathrm{SSYT}}
\def\SYT{\mathrm{SYT}}
\def\SetMT{\mathrm{SetMT}}
\def\Inc{\mathrm{Inc}}

\newcommand{\supp}{\mathrm{supp}}

\newcommand{\cL}{\mathcal{L}}

\def\fk{\mathfrak}

\def\barr{\begin{array}}
\def\earr{\end{array}}
\def\ba{\begin{aligned}}
\def\ea{\end{aligned}}
\def\be{\begin{equation}}
\def\ee{\end{equation}}

\def\Cyc{\mathrm{Cyc}}

\def\qquand{\qquad\text{and}\qquad}

\def\qquord{\qquad\text{or}\qquad}

\def\inv{\mathrm{Inv}}

\def\I{\mathcal{I}}

\def\DesR{\mathrm{Des}_R}
\def\DesL{\mathrm{Des}_L}
\def\Des{\mathrm{Des}}

\def\id{\mathrm{id}}
\def\PP{\mathbb{P}}

\def\fkS{\fk S}

\def\ben{\begin{enumerate}}
\def\een{\end{enumerate}}

\def\fpf{{\tt {FPF}}}

\def\DesI{\mathrm{Des}_V}

\def\ilambda{\mu}

\def\ellhat{\hat\ell}
\def\Ffpf{\hat F^\fpf}

\def\a{\textbf{a}}
\def\b{\textbf{b}}

\def\ic{\hat c}

\newcommand{\xRightarrow}[2][]{\ext@arrow 0359\Rightarrowfill@{#1}{#2}}

\newcommand{\Fl}{\operatorname{Fl}}

\renewcommand{\O}{\operatorname{O}}
\newcommand{\Sp}{\operatorname{Sp}}

\newcommand{\Ess}{\operatorname{Ess}}

\newcommand{\cA}{\mathcal{A}}

\def\iR{\hat\cR}

\def\F{\mathcal{F}}
\newcommand{\arc}[2]{ \ar @/^#1pc/ @{-} [#2] }
\def\arcstop{\endxy\ }
\def\arcstart{\ \xy<0cm,-.15cm>\xymatrix@R=.1cm@C=.3cm }
\newcommand{\arcstartc}[1]{\ \xy<0cm,-.15cm>\xymatrix@R=.1cm@C=#1cm}

\def\ellhat{\hat\ell}
\def\Ffpf{\hat F^\fpf}

\def\iS{\hat \fkS}
\def\iF{\hat F}
\def\iG{\hat G}

\newcommand{\ct}{\operatorname{ct}}

\newcommand{\pf}{\operatorname{pf}}

\def\iT{\hat{\fk T}}

\def\ipsi{\eta}

\newcommand{\minlex}{\operatorname{lt}}

\let\bxd\boxed
\renewcommand{\boxed}[1]{\hspace{0.7pt}\bxd{\!#1\!}\hspace{0.7pt}}

\def\Bump{{\tt Bump}}
\def\Insert{{\tt Insert}}
\def\SH{SH}
\def\Schensted{{\tt dir}}
\newcommand\ksim{\hspace{1mm}\hat\equiv\hspace{1mm}}

\numberwithin{equation}{section}
\allowdisplaybreaks[1]
\UseCrayolaColors

\makeatletter
\renewcommand{\@makefnmark}{\mbox{\textsuperscript{}}}
\makeatother

\begin{document}
\title{Schur $P$-positivity and involution Stanley symmetric functions}

\author{
    Zachary Hamaker \\
    Department of Mathematics \\
    University of Michigan \\
    { \tt zachary.hamaker@gmail.com}
 \and
    Eric Marberg \\
    Department of Mathematics \\
    HKUST \\
    {\tt eric.marberg@gmail.com}
\vspace{3mm}
\and
    Brendan Pawlowski\footnote{This author was partially supported by NSF grant 1148634.} \\
    Department of Mathematics \\
    University of Michigan \\
    {\tt br.pawlowski@gmail.com}
}

\date{}

\maketitle

\begin{abstract}
The \emph{involution Stanley symmetric functions} 
$\hat{F}_y$
are 
the stable limits of the 
analogues of Schubert polynomials
for the orbits of the orthogonal group in the flag variety.
These symmetric functions are also
generating functions for involution words, and
are indexed by the involutions in the symmetric group.
By construction each $\hat{F}_y$ is a sum of Stanley symmetric functions
 and therefore Schur positive. We prove the stronger fact that 
 these power series 
are Schur $P$-positive. 
We give an algorithm to efficiently compute the decomposition  
of $\hat{F}_y$ into Schur $P$-summands, and prove that this 
decomposition is triangular with respect to the dominance order on partitions.
As an application, we derive pattern avoidance conditions which characterize the involution
Stanley symmetric functions which are equal to Schur $P$-functions. 
We deduce as a corollary
 that the involution Stanley symmetric function of the reverse permutation
 is a Schur $P$-function indexed by a shifted staircase shape.
 These results
 lead to alternate proofs of theorems of Ardila--Serrano and DeWitt on
skew Schur functions which are Schur $P$-functions.  We also prove new Pfaffian formulas for
certain related involution Schubert polynomials.
 \end{abstract}

\tableofcontents
\setcounter{tocdepth}{2}

\section{Introduction}

In the seminal paper \cite{Stan}, Stanley defined 
for each permutation $w$ in the symmetric group $S_n$
 a certain symmetric function $F_w$. 
These symmetric functions are the stable limits of 
 Schubert polynomials, and so arise naturally in
the study of the geometry of the type A complete flag variety. 
They also occur in representation theory
as the characters of both generalized Schur modules
and the $U_q(A_n)$-crystals  introduced by Morse and Schilling in \cite{MorseSchilling}.
More concretely, these objects are
 useful to consider when 
counting the reduced words of permutations.
Stanley's construction was originally motivated as a tool for 
proving the following result:

\begin{theorem}[Stanley \cite{Stan}] \label{intro-thm1}
The cardinalities $(r_n)_{n \geq 1} = (1,1,2,16,768, 292864,\dots)$
of the set of
 reduced words 
for the reverse permutation $w_n = n\cdots 321 \in S_n$
satisfy
$r_n = \tbinom{n}{2}! \cdot 1^{1-n}\cdot 3^{2-n}\cdot 5^{3-n}\cdots (2n-3)^{-1}$
which  is the number of standard Young tableaux of shape $\delta_n=(n-1,\dots,2,1,0)$.
\end{theorem}

Let us explain how $F_w$ is related to the proof of this theorem. 
Let $s_i=(i,i+1) \in S_n $ for $i \in \{1,2,\dots,n-1\}$ and write $\cR(w)$ for the set of reduced words for $w \in S_n$, that is, 
the sequences of simple
transpositions $(s_{i_1}, s_{i_2},\dots,s_{i_\ell})$ of minimal possible length $\ell$ such that
$w = s_{i_1} s_{i_2}\cdots s_{i_\ell}$.
For an arbitrary sequence of simple transpositions $\textbf{a} = (s_{a_1}, s_{a_2},\dots, s_{a_\ell})$,
let $f_{\textbf{a}} \in \ZZ[[x_1,x_2,\dots]]$ denote the formal power series given by summing  the 
monomials $ x_{i_1} x_{i_2}\cdots x_{i_\ell}$
over all positive integers $i_1 \leq i_2 \leq \dots \leq i_\ell$ satisfying $i_j < i_{j+1}$ whenever $a_j < a_{j+1}$.

\begin{definition}\label{intro-def}
The \emph{Stanley symmetric function} of $w \in S_n$ is  
 $F_w = \sum_{\textbf{a} \in \cR(w)} f_{\textbf{a}}$.
 \end{definition}
 
 Our notation  differs from Stanley's in \cite{Stan} by an inversion of indices.
It is not obvious from this definition that $F_w$ is a symmetric function (for an alternate definition making this clear,
see Section~\ref{stab-sect}),
but it is evident that the size of $|\cR(w)|$ is the coefficient of $x_1x_2\cdots x_\ell$ in $F_w$,
where $\ell=\ell(w)$ is the length of $w$.
To find this coefficient we should expand $F_w$ in terms of one of the familiar bases of the algebra of symmetric functions. 
In general, this is difficult to do explicitly, but much can be said in special cases.
Recall that a permutation in $S_n$ is \emph{vexillary} if it is 2143-avoiding, 
and \emph{Grassmannian} if it has at most one right descent.

\begin{theorem}[Lascoux and Sch\"utzenberger \cite{LS}]\label{intro-ls-thm}
$F_w \in \NN\spanning\{ F_v : v \text { is Grassmannian}\}$.
\end{theorem}

\begin{theorem}[Stanley \cite{Stan}]
$F_w$ is a Schur function if and only if
 $w$ is vexillary.
\end{theorem}

Grassmannian permutations are vexillary, so these theorems imply the following corollary,
first proved by Edelman and Greene \cite{EG}, who also gave the first bijective proof of
Theorem~\ref{intro-thm1}.

\begin{corollary}[Edelman and Greene \cite{EG}]
Each $F_w$ is Schur positive. 
\end{corollary}

With a little more notation, one can make a stronger statement.
Write  $<$ for the dominance order on integer partitions. Combining results from \cite{LS, Stan}
gives the following.

\begin{theorem}[Stanley \cite{Stan}, Lascoux and Sch\"utzenberger \cite{LS}] \label{intro-a-thm}
Let $w \in S_n$ and let $a_j$ be the number of positive integers $i<j$ such that $w(i)>w(j)$.
If
 $\lambda$ is the transpose of the partition
given by sorting $(a_1,a_2,\dots,a_n)$,  then $F_w \in s_\lambda + \NN\spanning \{ s_\mu : \mu < \lambda\}$.
\end{theorem}

Coming full circle, the reverse permutation $w_n  \in S_n$ is certainly vexillary, 
and so the preceding theorem has this corollary, which implies Theorem~\ref{intro-thm1}
by the familiar hook length formula:

\begin{corollary}
If $n$ is a positive integer then $F_{w_n} = s_{\delta_n}$
for $\delta_n=(n-1,\dots,2,1,0)$.
\end{corollary}

We mention all of this as prelude to our main results,
which arise out of formally similar counting problems.
There are by now a multitude of generalizations \cite{BH,Lam2} of the symmetric functions $F_w$.
The one which will be of interest here comes from
the following construction for involutions in Coxeter groups.

Let $\I_n = \{ w \in S_n : w^2=1\}$ denote the set of involutions in $S_n$.
It is well-known (see Section~\ref{inv-schub-sect}) that there exists
a unique associative product $\circ : S_n \times S_n \to S_n$
such that $w\circ s_i = ws_i$ if $w(i) <w(i+1)$,  $w\circ s_i = w$ if $w(i) > w(i+1)$,
and $(v\circ w)^{-1} = w^{-1}\circ v^{-1}$.
For each $y \in \I_n$ define $\iR(y)$ as the set of sequences of simple
transpositions $(s_{i_1}, s_{i_2},\dots,s_{i_\ell})$ of minimal possible length $\ell$ such that
$
y =s_{i_\ell} \circ \cdots \circ s_{i_2} \circ  s_{i_1}\circ  s_{i_2}\circ \cdots\circ s_{i_\ell}.
$
Up to minor differences in notation,
the elements of $\iR(y)$ are the same as what Richardson and Springer \cite{RichSpring,RichSpring2}
call ``admissible sequences'', what Hultman \cite{H1,H2,H3} calls ``$\underline S$-expressions,
what Hu and Zhang \cite{HuZhang,HuZhang2} call ``$\textbf{I}_*$-expressions,''
and what we \cite{HMP1,HMP2,HMP3} have been calling \emph{involution words}.

There are a few different reasons why
one might consider this definition. Geometrically, the notion comes up
(e.g., in \cite{CJ,CJW,RichSpring}) when one
studies the action of the
orthogonal group $\O_n(\CC)$ on the flag variety. The orbits of this action are indexed by $\I_n$,
and the ``Bruhat order'' 
 induced by reverse inclusion of orbit closures coincides with the Bruhat order of $S_n$ restricted to
 $\I_n$.
In
representation theory, involution words arise in the study of
the Iwahori-Hecke algebra modules constructed by Lusztig and Vogan in \cite{LV1,LV2,LV3};
see, for example, the applications in
\cite{HuZhang,HuZhang2,Marberg}.
Finally, in combinatorics, these objects are interesting 
 in view of identities like the following, which we proved in \cite{HMP1}.
Note here that  $w_n = n \cdots321$ belongs to $\I_n$.

\begin{theorem}[See \cite{HMP1}] \label{our-thm1}
The numbers $(\hat r_n)_{n \geq 1} = (1,1,2,8,80,2688,\dots)$ giving the size of $\iR(w_n)$
satisfy
$\hat r_n = \tbinom{P+Q}{P} r_p r_q$
where 
$p=\lceil \frac{n+1}{2}\rceil$, $q= \lfloor \frac{n+1}{2}\rfloor$,
$P = \binom{p}{2}$, $Q = \binom{q}{2}$, and
$r_n$ is as in Theorem~\ref{intro-thm1}. 
\end{theorem}

This result shows that $\hat r_n$ is the number of standard bitableaux of shape
$(\delta_p, \delta_q)$, which is also  the dimension of the largest complex irreducible
 representation of the hyperoctahedral group of rank $P+Q$.
These numbers form a subsequence of \cite[A066051]{OEIS}.
To prove Theorem~\ref{our-thm1}, we introduced in \cite{HMP1} the following analogue of $F_w$:

\begin{definition}\label{intro-hatF-def}
The \emph{involution Stanley symmetric function} of $y \in \I_n$ is
$\iF_y = \sum_{\textbf{a} \in \iR(y)} f_{\textbf{a}}$.
\end{definition}

As with Definition~\ref{intro-def}, while it is evident that $|\iR(y)|$
can be extracted as a coefficient of $\iF_y$,
this formulation does not make clear that $\iF_y$ is a symmetric function,
or reveal the important fact that it is a multiplicity-free sum of
(ordinary) Stanley symmetric functions. 
(An alternate definition which indicates these properties appears in Section~\ref{inv-schub-sect}.)
These observations show that $\iF_y$ is manifestly Schur positive.
Our primary aim in this work is to prove that the symmetric functions $\iF_y$
have a stronger positivity property.

Within the ring of symmetric functions is the
 subalgebra $\QQ[p_1,p_3,p_5,\dots]$ generated by the odd power-sum functions.
 This algebra
 arises in few different
 places in the literature (e.g., \cite{BH, J16, P2,Schur, Stem}),
 and has a distinguished basis $\{ P_\lambda\}$ indexed by strict integer partitions
 (that is, partitions with all distinct parts),
whose elements $P_\lambda$ are called \emph{Schur $P$-functions}.
See Section~\ref{ss:schur-p} for the precise definition.
With this notation we can summarize our main results.
Define a permutation $y$ to be \emph{I-Grassmannian} if it has the form
$y = (\phi_1,m+1)(\phi_2,m+2)\cdots(\phi_r,m+r)$
for some positive integers $0 < \phi_1 < \phi_2 < \dots < \phi_r \leq m$.
In Section~\ref{igrass-sect}, we prove the following:

\begin{theorem}\label{our-thm2}
$\iF_y \in \NN\spanning\left\{ \iF_v : v \text { is I-Grassmannian}\right\}$.
\end{theorem}

Define $y \in \I_n$ to be \emph{$P$-vexillary} if $\iF_y$ is a Schur $P$-function.

\begin{theorem}\label{intro-ivex-thm}
There is a pattern avoidance condition characterizing $P$-vexillary involutions.
All I-Grassmannian involutions as well as the reverse permutations $w_n$ are $P$-vexillary.
\end{theorem}

This statement paraphrases Theorems~\ref{igrass-thm} and \ref{vex-thm}.
For the finite list of patterns that must be avoided, see Corollary~\ref{ivex-cor}.
In Section~\ref{i-vex-sect}, we use this  list to derive a new proof of a theorem of DeWitt \cite{DeWitt},
classifying the skew Schur functions which are Schur $P$-functions.
The last two theorems together imply the following:

\begin{corollary}\label{our-cor1}
Each $\iF_y$ is Schur $P$-positive, that is,
$\iF_y \in \NN\spanning \{ P_\lambda : \lambda\text{ is a strict partition}\}$.
\end{corollary}

In Section~\ref{tri-sect}, we prove the following analogue of Theorem~\ref{intro-a-thm}.
One can use this theorem to recover some results of Ardila and Serrano \cite{AS};
see Corollary~\ref{ardila-cor1}.

\begin{theorem}\label{intro-tri-thm}
Let $y \in \I_n$ and let $b_i$ be the number of positive integers $j $
with 
$j \leq  i < y(j)$ and $j < y(i)$.
If $\mu$ is the transpose of the partition given by sorting $(b_1,b_2,\dots,b_n)$,
then $\mu$ is strict  and
$\iF_y \in P_\mu + \NN\spanning\{P_\lambda : \lambda  <\mu\}$ where $<$ is the dominance
order on strict partitions.
\end{theorem}

Our proof of  Theorem~\ref{our-thm2} is constructive and, combined with the previous theorem,
gives an efficient algorithm for computing the expansion of any 
$\iF_y$ into Schur $P$-summands. This represents a massive generalization of our main results in
\cite{HMP1}, which computed $\iF_y$ in a rather limited special case, the most important example
of which occurs  when $y=w_n$. We can now derive a formula for $\iF_{w_n}$ as
an almost trivial corollary, from which Theorem~\ref{our-thm1} follows as a simple exercise:

\begin{corollary}\label{intro-wn-cor}
It holds that $\iF_{w_n} = P_{(n-1,n-3,n-5,\dots)} = s_{\delta_p} s_{\delta_q}$
for $p=\lceil \frac{n+1}{2}\rceil$ and $q= \lfloor \frac{n+1}{2}\rfloor$.
\end{corollary}

\begin{proof}
The first equality holds by Theorems~\ref{intro-ivex-thm} and \ref{intro-tri-thm}.
The second equality is a consequence of \cite[Theorem 1.4]{HMP1}
or \cite[Theorem 9.3]{Stem}.
\end{proof}

Our proofs of Theorems~\ref{our-thm2}, \ref{intro-ivex-thm}, and \ref{intro-tri-thm} are algebraic.
In Section~\ref{insertion-sect} we present a direct, bijective proof 
of Corollary~\ref{our-cor1}
 based on Patrias and Pylyavskyy's notion of
 \emph{shifted Hecke insertion} \cite{PP}. This alternate proof shows
 that 
 the coefficients in the Schur $P$-expansion of $\iF_y$ 
are the cardinalities of certain sets of shifted tableaux; see Corollary~\ref{t:insertion-cor}.

If we work with a  rescaled form of $\iF_y$,
then the results above may be reinterpreted in terms of the \emph{Schur $Q$-functions} $\{Q_\lambda\}$,
defined by $Q_\lambda = 2^{\ell(\lambda)} P_\lambda$ for strict partitions $\lambda$ with $\ell(\lambda)$ parts.
Let $\kappa(y)$ denote
the number of nontrivial cycles of an involution $y \in \I_n$ and define $\iG_y = 2^{\kappa(y)}\iF_y$. 
Clearly $\iG_y$ is a positive linear combination of Schur $Q$-functions, and 
in Section~\ref{schur-q-sect} we show that the coefficients which appear are actually positive integers.
Unlike the situation in Theorem~\ref{intro-tri-thm}, 
 the Schur $Q$-expansion of $\iG_y$, while still triangular with respect to dominance order, is no longer necessarily monic.

There is a noteworthy $Q$-analogue of  a vexillary permutation.
Define $y \in \I_n$ to be \emph{$Q$-vexillary} if $\iG_y$ is a Schur $Q$-function.
We prove the following in Section~\ref{schur-q-sect}.

\begin{theorem}\label{q-vex-thm}
An involution $y \in \I_n$ is $Q$-vexillary if and only if $y$ is vexillary, i.e., 2143-avoiding.
Every $Q$-vexillary involution is also $P$-vexillary, but not vice versa.
\end{theorem}

There is a  story parallel to all of this when we consider only fixed-point-free involutions in $S_{2n}$,
which parametrize
the orbits of the symplectic group
 $\Sp_{2n}(\CC)$ acting on the flag variety. 
There is a family of ``fixed-point-free'' involution Stanley symmetric functions $\Ffpf_z$,
 for which Theorems~\ref{our-thm2} and \ref{intro-ivex-thm} and most other results here have 
 interesting analogues.
The proofs of some these statements turn out to be significantly more
complicated than their predecessors.
 In order to keep the present article to a manageable length, we defer this material to \cite{HMP5}.

Here is an outline of what follows.
The proofs of our main results depend crucially on an interpretation of the symmetric functions
$\iF_y$ as stable limits of polynomials introduced by Wyser and Yong \cite{WY}
to represent the cohomology classes of certain orbit closures in the flag variety.
In \cite{HMP3}, we proved transition formulas for these cohomology representatives,
which we refer to as \emph{involution Schubert polynomials}.
After some preliminaries in Section~\ref{prelim-sect},
we review these transition formulas in Section~\ref{invtrans-sect}
and use them to derive some relevant identities for $\iF_y$.
We prove the theorems sketched in this introduction
in Section~\ref{schurp-sect}, and describe an alternate bijective proof of Schur $P$-positivity in Section~\ref{insertion-sect}.
Along the way, we also establish a few other results,
such as Pfaffian formulas
for certain involution Stanley symmetric functions and Schubert polynomials
(see Section~\ref{pfaffian-sect}).


\subsection*{Acknowledgements}

We thank
Dan Bump, Michael Joyce,
Vic Reiner,
Alex Woo,
Ben Wyser, and
Alex Yong for many helpful conversations during the development of this paper.

\section{Preliminaries}
\label{prelim-sect}

Let $\PP\subset \NN \subset \ZZ$ denote the respective sets of positive, nonnegative, and all 
integers. For $n \in \PP$, let $[n]=\{1,2,\dots,n\}$.
The \emph{support} of a permutation $w : X\to X$ is the set  $\supp(w) = \{ i \in X: w(i)\neq i\}$.
Define $S_\ZZ$ as the group of permutations of $\ZZ$ with finite support,
and let $S_\infty\subset S_\ZZ$ be the subgroup of permutations with support contained in $ \PP$.
We view $S_n$ as the subgroup of $S_\infty$ consisting of the permutations fixing all integers $i \notin [n]$.

Throughout, we let $s_i=(i,i+1) \in S_\ZZ$ for $i \in \ZZ$.
Let $\cR(w)$ be the set of reduced words for $w \in S_\ZZ$
and write $\ell(w)$ for the 
common length of these words. 
We let $\DesL(w)$ and $\DesR(w)$ denote the
\emph{left}  and \emph{right descent sets} of $w \in S_\ZZ$,
consisting of the simple transpositions $s_i$ such that $\ell(s_iw)<\ell(w)$ and
$\ell(ws_i)<\ell(w)$, respectively.  Recall that $s_i \in \DesR(w)$
if and only if $w(i)> w(i+1)$,
and that $\ell(w) $ is the cardinality of
$\inv(w) = \{ (i,j) \in \ZZ\times \ZZ : i<j\text{ and }w(i)>w(j)\}.
$
Let $<$ denote the \emph{(strong) Bruhat order} on $S_\ZZ$,
that is, the weakest  partial order on $S_\ZZ$ with $w < wt$ if $t $
is a transposition and $\ell(w)<\ell(wt)$. 
We write $u \lessdot v$ for $u,v \in S_\ZZ$ if $ \{ w \in S_\ZZ : u\leq w<v\}= \{u\}$.
The poset $(S_\ZZ, \leq)$ contains $S_\infty$ as a lower ideal and is graded with rank function
$\ell$. Consequently $u \lessdot v$  if and only if $u<v$ and $\ell(v) = \ell(u)+1$.
If $t=(a,b) \in S_\ZZ$ for integers $a<b$,
then $u \lessdot ut$ if and only if $u(a) < u(b)$ and no $i \in \ZZ$
exists with $a<i<b$ and $u(a)<u(i)<u(b)$.

\subsection{Divided difference operators}\label{divided-sect}

We recall a few technical facts about \emph{divided difference operators} from the references
\cite{Knutson, Macdonald, Macdonald2, Manivel}.
Let $\cL = \ZZ\left[x_1,x_2,\dots,x_1^{-1},x_2^{-1},\dots\right]$ be the ring of Laurent
polynomials over $\ZZ$ in a countable set of commuting indeterminates, and let
$\cP = \ZZ[x_1,x_2,\dots]$ be the subring of polynomials in $\cL$.
The group $S_\infty$ acts on $\cL$ by permuting variables, and one defines
\[ \partial_i f = (f - s_i f)/(x_i-x_{i+1})\qquad\text{for $i \in \PP$ and $f \in \cL$}.\]
The \emph{divided difference operator} $\partial_i$ defines a map $\cL \to \cL$ which restricts to
a map $\cP\to \cP$. It is clear by definition that $\partial_i f = 0$ if and only if $s_i f = f$.
If $f \in \cL$ is homogeneous and $\partial_i f \neq 0 $ then $\partial_i f $ is
homogeneous of degree $\deg(f)-1$.
If  $f,g \in \cL$  then $\partial_i(fg) =(\partial_if)g +  (s_if) \partial_i g$,
and if $\partial_i f = 0$, then
$\partial_i(fg) = f \partial_i g$.

%
%

The divided difference operators satisfy $\partial_i^2=0$ as well as the usual braid relations for
$S_\infty$, and so if $w \in S_\infty$ then $\partial_{i_1}\partial_{i_2}\cdots \partial_{i_k}$ is
the same map  $\cL\to \cL$  for all reduced words $(s_{i_1},s_{i_2},\dots,s_{i_k}) \in \cR(w)$.
We denote this map by
$\partial_w :\cL \to \cL$ for $w \in S_\infty.$
For $n \in \PP$, let $w_n = n\cdots 321\in S_n$ be the reverse permutation and define
$\Delta_n = \prod_{1\leq i<j \leq n}(x_i-x_j)$.
The following  identity is \cite[Proposition 2.3.2]{Manivel}.

\begin{lemma}[See \cite{Manivel}]\label{man-lem}
If $n \in \PP$ and $f \in \cL$ then
$\partial_{w_n} f = \Delta_n^{-1} \sum_{\sigma \in S_n} (-1)^{\ell(\sigma)} \sigma f$.
\end{lemma}

For $i \in \PP$ the \emph{isobaric divided difference operator} $\pi_i : \cL \to \cL$ is defined by
\[
\pi_i(f) = \partial_i(x_if) = f + x_{i+1}\partial_i f\qquad\text{for $f \in \cL$}.
\]
Observe that $\pi_i f=f$ if and only if $s_i f=f$, in which case $\pi_i(fg) = f \pi_i(g)$ for
$g \in \cL$.
If
$f \in \cL$ is homogeneous  with   $\pi_i f \neq 0 $, then $\pi_i f $ is homogeneous of the same
degree. The isobaric divided difference operators also satisfy the braid relations for $S_\infty$,
so we may define
$\pi_w=\pi_{i_1}\pi_{i_2}\cdots \pi_{i_k}$ for any $(s_{i_1},s_{i_2},\dots,s_{i_k}) \in \cR(w)$.
Moreover, $\pi_i^2=\pi_i$. 
%

Given $a,b \in \PP$ with $a< b$, define
$\partial_{b,a} = \partial_{b-1} \partial_{b-2}\cdots \partial_a$
and
$\pi_{b,a} = \pi_{b-1} \pi_{b-2}\cdots \pi_a$. For numbers $a,b \in \PP$ with $a\geq b$, we set
$\partial_{b,a} = \pi_{b,a} = \id$.
It is convenient here to note the following identity.

\begin{lemma}\label{pi-lem}
If $a\leq b$ and  $f \in \cL$ are such that $\partial_i f = 0$ for $a< i < b$,
then
$\pi_{b,a} f = \partial_{b,a}\(x_a^{b-a}f\).$
\end{lemma}

\begin{proof}
Assume $a<b$. By induction 
$\pi_{b,a} f = \pi_{b,a+1}\pi_{a} f
= \partial_{b,a+1}\( x_{a+1}^{b-a-1} \pi_a f\)
= \partial_{b,a}\(x_a^{b-a}f\) - \partial_{b,a+1}\(\partial_a\( x_a^{b-a-1}\) x_a f\)
= \partial_{b,a}\(x_a^{b-a}f\) - (x_af)  \partial_{b,a}\(x_a^{b-a-1}\).
$
Since $x_a^{b-a-1}$ has degree less than $b-a$, we have
$ \partial_{b,a}\(x_a^{b-a-1}\)=0$, so $\pi_{b,a}f = \partial_{b,a}(x_a^{b-a}f)$ as desired.
\end{proof}

For a sequence of integers  $a = (a_1,a_2,\dots)$ of either finite length or with only finitely many
nonzero terms, we let $x^a = x_1^{a_1} x_2^{a_2}\cdots \in \cL$. For $n \in \PP$, define
$\delta_n = (n-1,\dots,2,1,0)$.

\begin{lemma}\label{pi-lem2}
If $n \in \PP$ and $f \in \cL$ then $\pi_{w_n} f = \partial_{w_n}(x^{\delta_n} f)$.
\end{lemma}

\begin{proof}
Assume $n>1$ and let $c=s_1s_2\cdots s_{n-1} \in S_n$.
One checks that $\pi_{w_n} = \pi_c \pi_{w_{n-1}}$ and
$\partial_{w_n} = \partial_c\partial_{w_{n-1}}$ and
$\pi_c f = \partial_c (x_1x_2\cdots x_{n-1} f)$ for $f \in \cL$.
Hence, by induction,
$\pi_{w_n} f
= \pi_c\pi_{w_{n-1}} f
= \pi_c \partial_{w_{n-1}} (x^{\delta_{n-1}} f)
= \partial_c \partial_{w_{n-1}} (x_1x_2\cdots x_{n-1}x^{\delta_{n-1}} f)
=\partial_{w_n}(x^{\delta_n} f)$.
\end{proof}

\subsection{Schubert polynomials and Stanley symmetric functions}\label{stab-sect}

The \emph{Schubert polynomial}   corresponding to  $y \in S_n$ is the polynomial
$\fkS_y = \partial_{y^{-1}w_n}x^{\delta_n} \in \cP$, where as above  we let
$w_n=n\cdots321 \in S_n$ and $x^{\delta_n} = x_1^{n-1} x_2^{n-2}\cdots  x_{n-1}^1.$
This formula for $\fkS_y$ is independent of the choice of $n$
such that $y \in S_n$, and we consider
the Schubert polynomials to be a family indexed by $S_\infty$.
Some useful
references for the basic properties of $\fkS_w$ include \cite{BJS,Knutson,Macdonald,Manivel}.
Since $\partial_i^2=0$, it follows directly from the definition that
\be\label{maineq}
\fkS_1=1
\qquand
\partial_i \fkS_w = \begin{cases} \fkS_{ws_i} &\text{if $s_i \in \DesR(w)$}
\\
0&
\text{if $s_i \notin \DesR(w)$}\end{cases}
\qquad\text{for each $i \in \PP$}.
\ee
Conversely,
one can show that
$\{ \fkS_w\}_{w \in S_\infty}$ is the unique family of homogeneous polynomials indexed by
$S_\infty$ satisfying \eqref{maineq}; see \cite[Theorem 2.3]{Knutson} or the introduction of \cite{BH}.
One checks as an exercise that $\deg \fkS_w=\ell(w)$
and
$\fkS_{s_i} = x_1+x_2+\dots +x_i$ for $i \in \PP$.
The polynomials $\fkS_w$ for $w \in S_\infty$ are linearly independent, and form a $\ZZ$-basis for $\cP$ 
\cite[Proposition 2.5.4]{Manivel}.

Let $\ZZ[[x_1,x_2,\dots]]$ be the ring of formal power series of bounded degree in the commuting
variables $x_i$ for $i \in \PP$. 
Let $\Lambda\subset \ZZ[[x_1,x_2,\dots]]$ be the usual subring of \emph{symmetric functions}.
A sequence of power series $f_1,f_2,\dots $ has a limit $\lim_{n\to \infty} f_n\in \ZZ[[x_1,x_2,\dots]]$
  if for each fixed monomial the
corresponding coefficient sequence is eventually constant.
For  $n \in \NN$,  let $\rho_n : \ZZ[[x_1,x_2,\dots]] \to \ZZ[x_1,x_2,\dots,x_n]$ denote the
 homomorphism induced by setting $x_{i}=0$ for $i>n$.

\begin{lemma}\label{stab0-lem}
Let $f_1,f_2,\dots \in \cL$ be a sequence of homogeneous  Laurent polynomials and suppose for
some $N \in \NN$ it holds that $f_N\neq 0$ and that  $\rho_n f_{n+1} = f_n$  and
$x^{\delta_n} f_{n+1} \in \cP$ for all $n\geq N$.
Then $F = \lim_{n\to \infty} \pi_{w_n} f_n$ exists and belongs to $\Lambda$,
and satisfies $\rho_n F = \pi_{w_n} f_n$ for all $n\geq N$.
\end{lemma}

This lemma is false without a condition like homogeneity to control $\deg f_n$.

\begin{proof}
Since $\rho_n f_{n+1} = f_n$ for $n\geq N$ and since each $f_n$ is homogeneous, we must have
$f_n\neq 0$ and  $\deg f_n = \deg f_N$ for all $n\geq N$.
Note by Lemma \ref{man-lem} that if $n\geq N$ then $\pi_{w_n} f_n \in \cP$
is invariant under the action of $S_n$. As such,
to prove the lemma it suffices to show that
$\rho_n \pi_{n+1} f_{n+1} = \pi_n f_n$ for all $n\geq N$.
This is straightforward from Lemmas \ref{man-lem} and \ref{pi-lem2}
on noting that $\rho_n \Delta_{n+1} = x_1x_2\cdots x_n \Delta_n$ and that if $w \in S_{n+1}$
but $w \notin S_n$ for some $n\geq N$ then by hypothesis $\rho_n w (x^{\delta_{n+1}}f_{n+1}) =0$.
\end{proof}

\begin{corollary}\label{stab-cor}
If $p \in \cP$ is any polynomial then
$\lim_{n \to \infty} \pi_{w_n} p $ exists and belongs to $\Lambda$.
\end{corollary}

For Schubert polynomials, the limit in this corollary has a noteworthy alternate form.

\begin{definition}\label{shift-def}
For  $w \in S_\ZZ$ and $N\in \ZZ$, let $w \gg N \in S_\ZZ$ denote the map
$i\mapsto w(i-N) +N$. 
\end{definition}

Note that $\supp(w \gg N) = \{ i+N  : i \in \supp(w)\}$.
The following lemma is equivalent to  \cite[Eq.\ (4.25)]{Macdonald},
or to the combination of
\cite[Proposition 2.12, Corollary 3.38, and Theorem 3.40]{HMP1}.

\begin{lemma}[See \cite{Macdonald}]
\label{stab-lem}
If $w \in S_\infty$, $\DesR(w) \subset \{s_1,\dots, s_n\}$, and $N\geq n$, then
$\pi_{w_n} \fkS_w   = \rho_n \fkS_{w \gg N}$.
\end{lemma}

By Lemmas \ref{stab0-lem} and  \ref{stab-lem} we deduce the following:

\begin{thmdef}[See \cite{Macdonald,Stan}]\label{F-def}
If $w \in S_\ZZ$ then
$F_w = \lim_{N\to \infty} \fkS_{w\gg N}$ is a well-defined symmetric function,
which we refer to as the \emph{Stanley symmetric function} of $w$.
\end{thmdef}

It follows from results in \cite{BJS} that 
this definition gives the same power series as Definition~\ref{intro-def}.
It is clear 
that $F_w = F_{w\gg N}$ for any $N \in \ZZ$.
By Lemma \ref{stab-lem}, if $w \in S_\infty$ then $F_w = \lim_{n\to \infty} \pi_{w_n}\fkS_w$.

\subsection{Involution Schubert polynomials}\label{inv-schub-sect}

Let $\I_n$, $\I_\infty$, and $\I_\ZZ$ denote the sets of involutions  
in $S_n$,  $S_\infty$, and $S_\ZZ$.
The involutions in these groups are the permutations whose cycles all have at most two elements.
For $y \in \I_\ZZ$ define 
\[\Cyc_\ZZ(y) = \{ (i,j) \in \ZZ\times \ZZ : i \leq j = y(i)\}
\qquand
 \Cyc_\PP(y) = \Cyc_\ZZ(y)\cap (\PP\times \PP).\]
It is often convenient to identify  elements  of $\I_n$, $\I_\infty$, or $\I_\ZZ$ with the partial
matchings on  $[n]$, $\PP$, or $\ZZ$ in which
distinct vertices are connected by an edge whenever they form a nontrivial cycle.
By convention, we draw such matchings so that the vertices are points on a horizontal   axis, ordered from left to right,
and the edges appear as  convex curves in the upper half plane. For example,
\[
 (1,6)(2,7)(3,4) \in \I_7
 \qquad
 \text{is represented as}
 \qquad
\arcstart{
*{.}  \arc{1.}{rrrrr}   & *{.}  \arc{1.}{rrrrr}   & *{.} \arc{.5}{r} & *{.} & *{.} & *{.} & *{.}
}\endxy
\]
We often omit the numbers labeling the vertices in  matchings corresponding to
involutions in $\I_\infty$.

The next four propositions can all be recast as more general statements 
about twisted involutions in arbitrary Coxeter groups, and appear in this form
in \cite{RichSpring,RichSpring2} or \cite{H1,H2,H3,H4}.

\begin{propdef}
\label{demazure-lem}

There exists a unique associative product
$\circ : S_\ZZ \times S_\ZZ \to S_\ZZ$ such that
$u\circ v = uv$ if $\ell(uv) = \ell(u) +\ell(v)$ and $s_i\circ s_i = s_i$ for all $i \in \ZZ$.
\end{propdef}

Clearly $s\circ w = w\circ t  = w$ if $s \in \DesL(w)$ and $t \in \DesR(w)$.
If 
$(t_{1},t_{2},\dots,t_{k}) \in \cR(w)$  then
$w = t_{1} \circ t_{2} \circ \dots \circ t_{k} =
t_{1}  t_{2}  \dots  t_{k}  $.
The exchange principle
for Coxeter groups implies the following:

\begin{proposition}
If $y \in \I_\ZZ$ and $s=s_i$ then
$ s\circ y \circ s =
\begin{cases}
sys &  \text{if } s \notin\DesR(y) \text{ and }ys \neq sy \\
ys  &  \text{if } s \notin\DesR(y) \text{ and }sy=ys \\
y   &  \text{if } s \in \DesR(y).
\end{cases}
$
\end{proposition}

Thus, if $y \in \I_\ZZ$ then $s\circ y \circ s \in \I_\ZZ$,
and by induction on length one may deduce:

\begin{proposition}
If $y \in \I_\ZZ$ then $y = w^{-1} \circ w$  for some $w \in S_\ZZ$.
\end{proposition}

%

For $y \in \I_\ZZ$, let $\cA(y)$ denote the finite, nonempty set of permutations $w \in S_\ZZ$ of minimal length
such that $y = w^{-1}  \circ w$. 
The set $\iR(y)$ defined in the introduction is precisely $\iR(y) = \bigcup_{w \in \cA(y)} \cR(w)$.
We refer to the elements of $\cA(y)$
 as the \emph{atoms} of $y$. 
Let $\ellhat(y)$ denote the common length of each $w \in \cA(y)$.
One can show that $\ellhat =  \frac{1}{2}\( \ell+\kappa\)$,
where $\kappa(y)$ is the number of nontrivial cycles of $y \in \I_\ZZ$.

\begin{definition} \label{iS-def}
The \emph{involution Schubert polynomial} of $y \in \I_\infty$ is 
$ \iS_y = \sum_{w \in \cA(y)} \fkS_w.$
\end{definition}

\begin{example}\label{fkS-ex}
We have 
$\cA(321) = \{ 231, 312\}$, so
$\iS_{321} = \fkS_{231} + \fkS_{312} = x_1^2  + x_1 x_2.$
\end{example}

The essential algebraic properties of the polynomials $\iS_y$ are given by \cite[Theorem 3.11]{HMP1}:

\begin{theorem}[See \cite{HMP1}]\label{ithm}
The involution Schubert polynomials $\{ \iS_y\}_{y \in \I_\infty}$
are the unique family of homogeneous polynomials indexed by $\I_\infty$
such that if $i \in \PP$ and $s=s_i$ then
\be\label{i-eq}
\iS_{1} = 1
\qquand
\partial_i \iS_y =
\begin{cases}
\iS_{sys} &\text{if $s \in \DesR(y)$ and $sy\neq ys$} \\
\iS_{ys}  &\text{if $s \in \DesR(y)$ and $sy=ys$} \\
0         &\text{if $s \notin \DesR(y)$.}
\end{cases}
\ee
\end{theorem}

Observe that if $s_i \notin \DesR(y)$ then $\partial_i\iS_{s\circ y\circ s} = \iS_y$.
Since $\fkS_w$  has degree $\ell(w)$, it follows that $\iS_y$  has degree $\ellhat(y)$.
As the sets $\cA(y)$ for $y \in \I_\infty$ are pairwise disjoint,
the polynomials $\iS_y$ for $y \in \I_\infty$ are linearly independent.

The involution Schubert polynomials were introduced in a  rescaled form by Wyser and Yong
in \cite{WY}, where they were denoted $\Upsilon_{y; (\GL_n, \O_n)}$.  
The precise relationship is $2^{\kappa(y)} \iS_y = \Upsilon_{y; (\GL_n, \O_n)}$; see the discussion in \cite[Section 3.4]{HMP1}.
Wyser and Yong's definition was motivated by the study of the action of the orthogonal group
$\O_n(\CC)$ on the flag variety $\Fl(n) = \GL_n(\CC)/B$, with $B\subset \GL_n (\CC)$ denoting the
Borel subgroup of lower triangular matrices. It follows from \cite{WY} that
the involution Schubert polynomials $\iS_y$, rescaled by the factor $2^{\kappa(y)}$,
are cohomology representatives for the closures of the $\O_n(\CC)$-orbits in $\Fl(n)$,
and so are special cases of an older formula of Brion \cite[Theorem 1.5]{Brion98}.
See the discussion in \cite{HMP1, HMP3}.

The symmetric functions $\iF_y$ presented in the introduction are related to the polynomials $\iS_y$
by the following formula, which is equivalent to Definition~\ref{intro-hatF-def} 
since $\iR(y)= \bigcup_{w \in \cA(y)} \cR(w)$.

\begin{definition}\label{iF-def}
The \emph{involution Stanley symmetric function} of
$y \in \I_\ZZ$ is the power series 
$\iF_y = \sum_{w \in \cA(y)} F_w = \lim_{N\to \infty} \iS_{y \gg N} \in \Lambda.$
\end{definition}

The second equality in this definition holds by Theorem-Definition \ref{F-def}.
Note that $\iF_y$ is a homogeneous symmetric function of
degree $\ellhat(y)$.
If $y \in \I_\infty$, then $\iF_y = \lim_{n\to \infty} \pi_{w_n} \iS_y$.

\subsection{Schur $P$-functions}
\label{ss:schur-p}

Our main results will relate $\iF_y$ to the \emph{Schur $P$-functions} in $\Lambda$.
These symmetric functions were introduced  in work of Schur on the projective
representations of the symmetric group \cite{Schur} but have since
arisen in a variety of other contexts (see, e.g., \cite{BH, J16, P2}).
We briefly review some of their properties from
\cite[\S6]{Stem} and \cite[{\S}III.8]{Macdonald2}.
For integers  $0\leq r \leq n$, let
\[
G_{r,n} = \prod_{i \in [r]} \prod_{j \in [n-i]} \(1 +x_i^{-1} x_{i+j}\) \in \cL.
\]
For a partition $\lambda = (\lambda_1,\lambda_2,\dots)$, let $\ell(\lambda)$
denote the largest index $i \in \PP$ with $\lambda_i\neq 0$.
The partition $\lambda$ is \emph{strict} if $\lambda_i \neq \lambda_{i+1}$
for all $i < \ell(\lambda)$.
Recall that $x^\lambda = x_1^{\lambda_1}x_2^{\lambda_2}\cdots x_{r}^{\lambda_r}$
for $r = \ell(\lambda)$.

\begin{definition}\label{schurp-def}
For a strict partition $\lambda$
with $r=\ell(\lambda)$ parts, 
let $P_\lambda =  \lim_{n\to \infty} \pi_{w_n} \(x^\lambda G_{r,n}\) \in \Lambda$.
The symmetric function $P_\lambda$ is the \emph{Schur $P$-function} corresponding to $\lambda$.
\end{definition}

By Lemma \ref{stab0-lem}, this formula for $P_\lambda$ gives a well-defined,
homogeneous symmetric function of degree $\sum_i \lambda_i$,
and   $\rho_n P_\lambda = \pi_{w_n} \(x^\lambda G_{r,n}\) $ for $n\geq r = \ell(\lambda)$.
We emphasize this definition of $P_\lambda$ for its
compatibility with our definition of $F_w$ in Section \ref{stab-sect}.
One can show that the Schur function
 $s_\lambda$ is given by a similar limit: namely, $s_\lambda = \lim_{n\to\infty} \pi_{w_n} x^\lambda$.

Some other similarities exist between  $s_\lambda$ and $P_\lambda$.
Whereas the Schur functions form a $\ZZ$-basis for $\Lambda$, the Schur $P$-functions
form a $\ZZ$-basis for the subring  $\QQ[ p_1,p_3,p_5,\dots ] \cap \Lambda$
generated by the odd-indexed power sum symmetric functions \cite[Corollary 6.2(b)]{Stem}.
Each Schur $P$-function $P_\lambda$ is itself Schur positive \cite[Eq.\ (8.17), {\S}III.8]{Macdonald2}.

The symmetric functions $P_\lambda$ may be described more concretely as generating functions for certain shifted tableaux.
We review this perspective in Section~\ref{ss:schur-p2}.

\section{Transition formulas}\label{invtrans-sect}

In this section, we review the transition formula for $\iS_y$
proved in \cite{HMP3}.
This result is  similar to the following  identity for $\fkS_w$.
Given $y \in S_\ZZ$ and $r \in \ZZ$, define $\Phi^\pm(y,r)$ as the set of permutations $w \in S_\ZZ$
such that $\ell(w) = \ell(y)+1$ and $w= y(r,s)$ for some $s \in \ZZ$ with $\pm(r-s) < 0$.

\begin{theorem}[See \cite{Knutson}]
\label{transition-thm}
If $y \in S_\infty$ and
$r \in \PP$ then $x_r \fkS_y   = \sum_{ w \in \Phi^+(y,r)} \fkS_{w} -  \sum_{w \in \Phi^-(y,r)} \fkS_{w}$
where we set $\fkS_w = 0$ for $w \in S_\ZZ - S_\infty$.
\end{theorem}

%
%
This formula appears, for example,
as \cite[Corollary 3.3]{Knutson},
and is equivalent to \emph{Monk's rule} (see \cite[\S2.7]{Manivel}).
Taking limits transforms this 
to the following identity, which is \cite[Theorem 3]{Little}.

\begin{theorem}
\label{transition-thm2}
If $y \in S_\ZZ$ and $r \in \ZZ$ then $\sum_{ w \in \Phi^-(y,r)} F_{w} =  \sum_{w \in \Phi^+(y,r)} F_{w}$.
\end{theorem}


To state a transition formula for the involution Schubert polynomials $\iS_y$,
 we need to review a few technical properties of the Bruhat order  $<$ on $S_\ZZ$ restricted to $\I_\ZZ$.
Our notation follows Section~\ref{inv-schub-sect}.
More general results of Hultman 
imply the following useful facts:

\begin{theorem}[Hultman \cite{H1,H2,H3}] \label{invbruhat-thm}
The following properties  hold:
\ben
\item[(a)] $(\I_\ZZ,<)$ is a graded poset with rank function $\ellhat$.

\item[(b)] Fix $y,z \in \I_\ZZ$
and $w \in \cA(z)$. Then $y\leq z$ if and only if there exists $v \in \cA(y)$ with $v \leq w$.

\een
\end{theorem}

We write $y\lessdot_\I z$ if $z \in \I_\ZZ$ covers $y \in \I_\ZZ$ in the partial order
given by restricting $<$ to $\I_\ZZ$, that is, if $\{ w \in \I_\ZZ : y \leq w < z\} = \{y\}$.
While $y \lessdot_\I z $ $\Rightarrow$ $y<z$ and
$y \lessdot z$ $\Rightarrow $ $y \lessdot_\I z$, it does not hold that
$y \lessdot_\I z$ $\Rightarrow$ $y \lessdot z$ for $y,z \in \I_\ZZ$.
Given a finite set $E\subset \ZZ$ of size $n$,
write $\phi_E$ and $\psi_E$ for the unique order-preserving bijections
$\phi_E : [n] \to E$ and $\psi_E : E \to [n]$.
For $w \in S_\ZZ$,  define
\be\label{standard-eq}
[w]_E = \psi_{w(E)} \circ w \circ \phi_E \in S_n\subset S_\ZZ.
\ee
We call $[w]_E$ the \emph{standardization} of $w$ with respect to $E$.
This notation is intended to distinguish $[w]_E $ from the restriction of $w$ to $E$,
which we instead denote as
$
w|_E : E \to \ZZ.
$
The following is a consequence of \cite[Theorem 1.3]{HMP3}
and
the
results in \cite[Section 3]{HMP3}.

\begin{thmdef}[See \cite{HMP3}]
Let $y \in \I_\ZZ$ and choose integers $i<j$.
\ben
\item[(a)]
There exists at most one $z \in \I_\ZZ$
such that  $ \{ w \in \cA(y) : w \lessdot w(i,j) \in \cA(z)\} $ is nonempty.

\item[(b)] If $\{i,j,y(i),y(j)\}=[n]$ for an integer $n \in\{2,3,4\}$, then 
   define $\tau_{ij}(y)$ to be the involution $z$ in part (a) or set $\tau_{ij}(y) =y$ if no such $z$ exists.
   
\item[(c)]
In all other cases, define $\tau_{ij}(y) \in \I_\ZZ$ to be the unique permutation with
$  [\tau_{ij}(y)]_A =  \tau_{ab}([y]_A) $ and $  \tau_{ij}(y)|_B = y|_B$,
where 
$A = \{i,j,y(i),y(j)\}$, $B=\ZZ\setminus A$,
 $a = \psi_A(i)$, and $b = \psi_A(j)$.
 It still holds that if an involution $z$ exists as in (a), then $\tau_{ij}(y) = z$.
\een
\end{thmdef}

\begin{remark}
It always holds that $y \leq \tau_{ij}(y)$. If $\{i,j,y(i),y(j)\}$ is a consecutive set
and $y < \tau_{ij}(y)$, then $y \lessdot \tau_{ij}(y)$. In general, however, it can happen
that $y < \tau_{ij}(y)$ but $\ellhat(\tau_{ij}(y)) - \ellhat(y) > 1$.
\end{remark}

To compute $\tau_{ij}(y)$ in general,
we only need to know a formula for $\tau_{ij}(y)$ when $y \in \I_n$ and
$[n] = \{i,j,y(i),y(j)\}$.
This information is  specified in Table~\ref{ct-fig}.

\begin{table}[h]
\[
\barr{| c | c | c | c | c | l}
\hline&&&&\\
A=\{i,j,y(i),y(j)\} & [y]_A & (i,j) & [\tau_{ij}(y)]_A
& \sigma\text{ such that }\tau_{ij}(y)=y\sigma
\\&&&&\\
\hline
&&&&\\
\{a<b\}
&
\arcstart
{
*{.}     & *{.}
}
\arcstop
&(a,b)&
\arcstart
{
*{.}   \arc{.6}{r}  & *{.}
}
\arcstop
&
(a,b)
\\&&&&\\
\hline
&&&&\\
\{a<b<c\}
&
 \arcstart
{
*{.}    \arc{.6}{r}  & *{.} & *{.}
}
\arcstop
&
(b,c), (a,c)
&
 \arcstart
{
*{.}    \arc{.8}{rr}  & *{.} & *{.}
}
\arcstop
&
(a,c,b)
\\ &&&& \\
&
\arcstart
{
*{.}     & *{.}  \arc{.6}{r} & *{.}
}
\arcstop
&
(a,b),(a,c)
&
 \arcstart
{
*{.}    \arc{.8}{rr}  & *{.} & *{.}
}
\arcstop
&
(a,b,c)
\\&&&&\\
\hline&&&&\\
\{a<b<c<d\}
&
  \arcstart
{
*{.}  \arc{.6}{r}   & *{.}    & *{.} \arc{.6}{r}& *{.}
}
\arcstop
&
(b,c)
&
\arcstart
{
*{.}  \arc{.8}{rr}   & *{.} \arc{.8}{rr}   & *{.} & *{.}
}
\arcstop
&
(a,d)(b,c)
\\ &&&& \\
&
  \arcstart
{
*{.}  \arc{.6}{r}   & *{.}    & *{.} \arc{.6}{r}& *{.} 
}
\arcstop
& (a,c),(b,d),(a,d)
&
\arcstart
{
*{.}  \arc{.8}{rrr}   & *{.}    & *{.} & *{.} 
}
\arcstop
&
(a,c,d,b)
\\ &&&& \\
&
  \arcstart
{
*{.}  \arc{.8}{rr}   & *{.} \arc{.8}{rr}   & *{.} & *{.} 
}
\arcstop
& (a,b),(c,d) , (a,d)
&
\arcstart
{
*{.}  \arc{.8}{rrr}   & *{.}  \arc{.4}{r}  & *{.} & *{.} 
}
\arcstop
&
(a,b)(c,d)
\\
&&&&\\\hline
\earr
\]
\caption{Nontrivial values of $\tau_{ij}(y)$.
Fix $y \in \I_\ZZ$ and $i<j$ in $\ZZ$, and define $A = \{i,j,y(i),y(j)\}$.
The first column labels the elements of $A$.
The third column rewrites $(i,j)$ in this labeling.
The second and fourth columns identify the matchings which represent $[y]_A$ and $[\tau_{ij}(y)]_A$. 
For  values of $y$ and $i<j$  that do not correspond to  any rows in this table,
we have defined $\tau_{ij}(y) = y$.
}\label{ct-fig}
\end{table}

\begin{example}
If $y = (1,9)(2,4)(3,7)(5,10) \in \I_{10}$ then
$\tau_{2,10}(y) = (1,9)(2,10)(3,7)$,
 that is:
 \smallskip
\[
\tau_{2,11}\Bigr(
\arcstart
{
*{.}  \arc{1.2}{rrrrrrrr}   & *{.}  \arc{.6}{rr}
 & *{.} \arc{.8}{rrrr}
&
*{.}    & *{.} \arc{1.0}{rrrrr}
&
*{.}   & *{.} 
&
*{.}     & *{.} 
&
*{.}
}
\arcstop\Bigl)
 =
\arcstart
{
*{.}  \arc{1.2}{rrrrrrrr}   & *{.}  \arc{1.2}{rrrrrrrr}
   & *{.} \arc{.8}{rrrr}
&
*{.}    & *{.}
&
*{.}   & *{.}
&
*{.}     & *{.}
&
*{.}
}
\arcstop
\]
\end{example}

Apart from some differences in notation,
the map $\tau_{ij} : \I_\ZZ \to \I_\ZZ$ is essentially the same as the map $\ct_{ij} $ which Incitti
defines in
\cite{Incitti1}; see the discussion in \cite[Section 3.1]{HMP3}.
Incitti's work implies the following theorem, which we also stated as \cite[Theorem 3.16]{HMP3}.

\begin{theorem}[Incitti \cite{Incitti1}] \label{taubruhat-thm}
Let $y,z \in \I_\ZZ$. The following are then equivalent:
\ben
\item[(a)] $y\lessdot_\I z$.
\item[(b)] $z = \tau_{ij}(y)$ for some $i<j$ in $\ZZ$ and  $\ellhat(z) = \ellhat(y)+1$.
\item[(c)] $z = \tau_{ij}(y)$ for some $i<j$ in $\ZZ$ with $y(i)\leq i$ and $y\lessdot y(i,j)$.
\item[(d)] $z = \tau_{ij}(y)$ for some $i<j$ in $\ZZ$ with $j\leq y(j)$ and $y\lessdot y(i,j)$.
\een
\end{theorem}

Now,
given $y \in \I_\ZZ$ and $r \in \ZZ$, we define
\[
\ba
\hat\Phi^+(y,r) &=
\left\{
z  \in \I_\ZZ : \ellhat(z) = \ellhat(y)+1\text{ and }z = \tau_{rj}(y) \text{ for an integer }j>r
\right\}
\\
\hat\Phi^-(y,r) &=
\left\{
z  \in \I_\ZZ : \ellhat(z) = \ellhat(y)+1\text{ and }z = \tau_{ir}(y) \text{ for an integer }i<r
\right\}
.
\ea
\]
These sets are both nonempty  \cite[Proposition 3.26]{HMP3},
and  if $z \in \hat\Phi^\pm(y,r)$ then $y \lessdot_\I z$.
Moreover, Theorem~\ref{taubruhat-thm} implies that if $(p,q) \in \Cyc_\ZZ(y)$ then the following holds:
\begin{itemize}
\item[1.] If $ q<j $  and $z = \tau_{qj}(y)$,
then
$z \in \hat\Phi^+(y,q)$ if and only if $y \lessdot y(q,j)$.

\item[2.] If  $i<p$ and $z=\tau_{ip}(y)$,
then
$z \in \hat\Phi^-(y,p)$ if and only if  $y\lessdot y(i,p)$.
\end{itemize}
For $(p,q) \in \PP\times \PP$,  let $x_{(p,q)} = x_p=x_q$ if $p=q$
and otherwise set $x_{(p,q)} = x_p+x_q$.
The following transition formula for involution Schubert polynomials is \cite[Theorem 3.28]{HMP3}.

\begin{theorem}[See \cite{HMP3}]\label{invmonk-thm}
If $y \in \I_\infty$ and $(p,q) \in \Cyc_\PP(y)$
then $x_{(p,q)} \iS_y = \sum_{ z \in \hat\Phi^+(y,q)} \iS_{z} -  \sum_{z \in \hat\Phi^-(y,p)} \iS_{z}$
where we set  $\iS_z = 0$ for all $z \in \I_\ZZ-\I_\infty$.
\end{theorem}

\begin{example}\label{phihat-ex}
If $y=(2,3)(4,7) \in \I_7$ then
\[
\ba
\hat\Phi^+(y,3)
&= \{ \tau_{3,4}(y),\ \tau_{3,5}(y),\ \tau_{3,7}(y)\}
= \{ (2,4)(3,7),\ (2,5)(4,7),\ (2,7)\}
\\
\hat\Phi^-(y,2) &= \{\tau_{1,2}(y)\} = \{ (1,3)(4,7)\}
\ea
\]
so
$(x_2+x_3)\iS_{(2,3)(4,7)} = \iS_{(2,4)(3,7)}+\iS_{(2,5)(4,7)}+\iS_{(2,7)}-\iS_{(1,3)(4,7)}$.
\end{example}

Our new results will depend on the following identity.

\begin{theorem}\label{invmonk-thm2}
If $y \in \I_\ZZ$ and $(p,q) \in \Cyc_\ZZ(y)$ then
$ \sum_{z \in \hat\Phi^-(y,p)} \iF_{z}  = \sum_{ z \in \hat\Phi^+(y,q)} \iF_{z}.$
\end{theorem}

\begin{proof}
It holds that $\hat \Phi^\pm(y \gg N, r+N) = \{ w \gg N : w \in \hat \Phi^\pm(y,r)\}$
for all $y \in \I_\ZZ$ and $r,N \in \ZZ$.
By Theorem~\ref{invmonk-thm},
it follows that 
$\sum_{ z \in \hat\Phi^+(y,q)} \iF_{z} -  \sum_{z \in \hat\Phi^-(y,p)} \iF_{z}
= \lim_{N\to \infty} x_{(p+N,q+N)} \iS_{y\gg N} = 0$.
\end{proof}

\section{Positivity for involution Stanley symmetric functions}\label{schurp-sect}

%
In this section we prove our main results about the positive expansion of   $\iF_y$ into Schur $P$-functions.

\subsection{I-Grassmannian involutions}\label{igrass-sect}

Recall from \cite{Knutson,Manivel}
that 
the \emph{diagram} of $w\in S_\infty$ is the set $D(w) = \{ (i,w(j)): (i,j)\in \inv(w)\}$.
We orient the elements of $D(w)$ like the positions in a matrix.
The \emph{code} of $w \in S_\infty$ is the sequence
$c(w) = (c_1,c_2,c_3,\dots)$ in which $c_i$ is the number of positions in the $i$th row of $D(w)$.
Of course, $D(w)$ is a finite set and $c(w)$ has only finitely many nonzero terms. We make no
distinction between $(a_1,a_2,\dots,a_k)$ and the infinite sequence
$(a_1,a_2,\dots,a_k,0,0,\dots)$.
The \emph{essential set}  of  $D\subset \PP\times \PP$ is  the set $\Ess(D)$ of
 positions $(i,j) \in D$ such that $(i+1,j) \notin D $
and $(i,j+1) \notin D$.

\begin{example}
 If $w = 4231$ then $\Ess(D(w)) = \{ (3,1),(1,3)\}$ and $ c(w) = (3,1,1)$.
 \end{example}

\begin{definition}\label{grass-def}
A permutation $w \in S_\ZZ$ is \emph{Grassmannian} if $|\DesR(w)|\leq 1$.
\end{definition}

The proof of the next statement is an instructive exercise; see, e.g., \cite[Chapter 2]{Manivel}.

\begin{propdef}\label{propdef-1}
For $w \in S_\infty$ and $n \in \PP$, the following are equivalent:
\ben
\item[(a)] $\DesR(w) = \{ s_n\}$, i.e.,  $w(1)<w(2)<\dots <w(n) > w(n+1)< w(n+2)<\cdots $.

\item[(b)] $c(w)=(c_1,c_2,\dots,c_n,0,0,\dots)$ where $c_1\leq c_2 \leq \cdots \leq c_n\neq 0$.

\item[(c)] $\Ess(D(w))$ is nonempty and contained in $\{ (n,j) : j\in \PP\}$.
\een
A permutation $w \in S_\infty$ with these equivalent properties is called \emph{$n$-Grassmannian}.
The identity $1 \in S_\infty$ is by convention the unique \emph{0-Grassmannian} permutation.
\end{propdef}

Let $\lambda(w) = (w(n)-n,\dots,w(2)-2, w(1)-1)
=(c_n,\dots,c_2,c_1)$ for an $n$-Grassmannian permutation $w \in S_\infty$ with code
$c(w) = (c_1,c_2,\dots)$. Also define $\lambda(1) = \emptyset=(0,0,\dots)$.
The map $w \mapsto \lambda(w)$ is a bijection from $n$-Grassmannian permutations in $S_\infty$ to
partitions with at most $n$ parts.
Recall the definition of the map $\rho_n : \ZZ[[x_1,x_2,\dots]] \to \ZZ[x_1,x_2,\dots,x_n]$
from Section~\ref{stab-sect}.
The main object of this section is to prove an involution analogue of the following theorem.

\begin{theorem}[See \cite{Manivel}]
\label{classical-grass-thm1}
If $w \in S_\infty$ is $n$-Grassmannian, then $\fkS_w = \rho_n s_{\lambda(w)}$ and
$F_w = s_{\lambda(w)}$.
\end{theorem}

The goal of this section to identify a class of involutions for which a similar result holds.
To this end, consider the following variations of  $D(w)$ and $c(w)$, introduced in
\cite[Section 3.2]{HMP1}:

\begin{definition}
The \emph{(involution) diagram} of $y\in \I_\infty$ is the set
$\hat D(y) = \{ (i,j) \in  D(y) : j \leq i\}$. Equivalently, $(i,j ) \in \PP\times \PP$
belongs to $\hat D(y)$ if and only if $ j \leq i < y(j)$ and $j < y(i)$.
\end{definition}

\begin{definition}
The \emph{(involution) code} of $y\in \I_\infty $ is the sequence
$\hat c(y) = (c_1,c_2,\dots)$ in which $c_i$ is the number of positions
in the $i$th row of $\hat D(y)$.
\end{definition}

\begin{example} If $y=(1,4)$ then $\hat D(y) = \{ (1,1),(2,1),(3,1) \}$ and $\hat c(y) = (1,1,1)$.
\end{example}

Note that $|\hat D(y)|=\ellhat(y)$ \cite[Proposition 3.6]{HMP1}.
An involution $y\in \I_\infty$ is \emph{dominant} if  $\hat D(y)$
is the transpose of the shifted diagram of a strict partition (see Section~\ref{ss:schur-p2}),
which occurs if and only if $y$ is 132-avoiding
\cite[Proposition 3.25]{HMP1}. 
Recall that $x_{(i,j)}$ is either $x_i=x_j$ (if $i=j$) or $x_i+x_j$ (if $i\neq j$).
The following is \cite[Theorem 3.26]{HMP1}.

\begin{theorem}[See \cite{HMP1}]
\label{dominant-thm}
If $y \in \I_\infty$
is dominant then
$
\iS_y =  \prod_{(i,j) \in \hat D(y)} x_{(i,j)}
.$
\end{theorem}

The \emph{lexicographic order} on $S_\infty$ is the total order
induced by identifying $w \in S_\infty$ with its one-line representation $w(1)w(2)w(3)\cdots$.
Denote the lexicographically minimal element of $\cA(y)$ as $\alpha_{\min}(y)$.
The following statement is part of \cite[Theorem 6.10]{HMP2}.

\begin{lemma}[See \cite{HMP2}] \label{minatom-lem}
Suppose $y \in \I_\infty$ and $\Cyc_\PP(y) = \{ (a_i,b_i) : i \in \PP\}$ where $a_1<a_2<\cdots$.
The lexicographically  minimal element  $\alpha_{\min}(y) \in \cA(y)$  is the inverse of the permutation
whose one-line representation is given by  the sequence
 $b_1a_1b_2a_2b_3a_3\cdots$ with $a_i$ omitted whenever $a_i=y(a_i) =b_i$.
\end{lemma}

\begin{example}
If $y = (1,4)$ then $b_1a_1b_2a_2b_3a_3= 412233$ and
$\alpha_{\min}(y)=4123^{-1} = 2341$.
\end{example}

We say that a pair $(i,j) \in \ZZ$ is a \emph{visible inversion} of $y \in \I_\ZZ$ if $i<j$ and
$y(j) \leq \min\{ i, y(i)\}$.

\begin{lemma}\label{desi-lem0}
The set of visible inversions of $y \in \I_\infty$ is equal to $\inv(\alpha_{\min}(y))$.
\end{lemma}

\begin{proof}
Fix $y \in \I_\infty$ and let $\Cyc_\PP(y) = \{ (a_i,b_i) : i \in \PP\}$ where $a_1<a_2<\cdots$. 
All visible inversions of $y$ are contained in $\PP\times \PP$.
Let $m<n$ be positive integers and
let $j,k \in \PP$ be such that $m \in \{a_k, b_k\}$ and $n \in \{a_j,b_j\}$.
By Lemma \ref{minatom-lem}, we have $(m,n) \in \inv(\alpha_{\min}(y))$
if and only if $j \leq k$, which holds if and only if $a_j \leq a_k$.
Note that $a_j = \min\{n,y(n)\}$ and $a_k = \min\{m,y(m)\}$.

If $(m,n)$ is a visible inversion of $y$ then $y(m)>y(n)$ and $n = b_j>m\geq a_j = y(n)$, so
$a_j\leq  \min \{m,y(m)\}= a_k$ as desired.
Conversely, if $a_j \leq a_k$, then $n\neq a_j$ since $a_k \leq m$, so
 we must have $y(n) = a_j \leq a_k = \min\{m,y(m)\}$ which means that $(m,n)$
 is a visible inversion of $y$.
\end{proof}

The preceding lemma implies the following result, which is also  \cite[Lemma 3.8]{HMP1}.

\begin{lemma}[See \cite{HMP1}]
\label{codei-lem}
If $y \in \I_\infty$ then
$\hat c(y) = c(\alpha_{\min}(y))$.
\end{lemma}

We say that $i \in \ZZ$ is a \emph{visible descent} of $y\in \I_\ZZ$ if $(i,i+1)$
is a visible inversion, and
define  $\DesI(y) = \{ s_i : i \in \ZZ \text{ is a visible descent of $y$}\}$.
We note two facts about this set.

\begin{lemma}\label{desi-lem}
If $y \in \I_\infty$ then $\DesI(y) = \DesR(\alpha_{\min}(y))$.
\end{lemma}

\begin{proof}
This follows from Lemma \ref{desi-lem0} since $s_i \in \DesR(w)$ if and only if
$(i,i+1) \in \inv(w)$.
\end{proof}

\begin{lemma}\label{desi-lem2}
If $y \in \I_\infty$
then the $i$th row of  $\Ess(\hat D(y))$ is nonempty if and only if $s_i \in \DesI(y)$.
\end{lemma}

\begin{proof}
If $s_i \in \DesI(y)$
then $(i,y(i+1)) \in \hat D(y)$ but all positions of the form $(i+1,j) \in \hat D(y)$
have $j<y(i+1)$, so the   $i$th row of $ \Ess(\hat D(y))$ is nonempty.
Conversely, if the $i$th row of $ \Ess(\hat D(y))$ is nonempty,
then there exists $(i,j) \in \hat D(y)$ with $(i+1,j) \notin \hat D(y)$.
This occurs if and only if $j=y(k)$ for some $k >i$ with $y(i)>y(k)$ and
$i\geq y(k)\geq y(i+1)$, in which case evidently
$s_i \in \DesI(y)$.
\end{proof}

We may now give analogues of
Definition \ref{grass-def} and
Proposition-Definition \ref{propdef-1}.

\begin{definition} \label{igrass-def}
An involution $y \in \I_\ZZ$ is \emph{I-Grassmannian} if $|\DesI(y)|\leq 1$.
\end{definition}

For  $y \in \I_\infty$, this definition is 
equivalent to the one 
 in the introduction by the following.

\begin{propdef}\label{propdef-2} For   $y \in \I_\infty$ and $n \in \PP$,
the following are equivalent:
\ben
\item[(a)] $\DesI(y) = \{ s_n\}$.

\item[(b)] $y = (\phi_1, n+1)(\phi_2,n+2)\cdots (\phi_r,n+r)$
for integers $r \in \PP$ and  $1\leq \phi_1 < \phi_2<\dots < \phi_r \leq n$.

\item[(c)] $\hat c(y)=(c_1,c_2,\dots,c_n,0,0,\dots)$ where
$c_1\leq c_2 \leq \cdots \leq c_n\neq 0 $.

\item[(d)] $ \Ess(\hat D(y))$ is nonempty and contained in $\{ (n,j) : j\in \PP\}$.

\item[(e)] The lexicographically minimal atom $\alpha_{\min}(y) \in \cA(y)$ is $n$-Grassmannian.
\een
We refer to involutions $y \in \I_\infty$ with these equivalent properties  as \emph{$n$-I-Grassmannian},
and consider  $1 \in \I_\infty$ to be the unique \emph{0-I-Grassmannian} involution.
\end{propdef}

\begin{proof}
The equivalences (a) $\Leftrightarrow$ (c) $\Leftrightarrow$ (d) $\Leftrightarrow$ (e)
follow from Proposition-Definition \ref{propdef-1} and Lemmas \ref{codei-lem}, \ref{desi-lem},
and \ref{desi-lem2}.
Proving that (a) $\Leftrightarrow$ (b) is a straightforward exercise from the definitions.
\end{proof}

\begin{remark}
The number 
 $g_n$ of I-Grassmannian elements of $\I_n$
 satisfies  
$g_n = g_{n-1} + g_{n-2} + n-2$ for $n \geq 3$. 
The sequence $(g_n)_{n\geq 1} = (1, 2, 4, 8, 15, 27, 47, 80,\dots)$ 
appears as \cite[A000126]{OEIS}.
\end{remark}

Any involution in $S_\ZZ$ which is Grassmannian in the ordinary sense is also I-Grassmannian.
Moreover, $y \in \I_\ZZ$ is I-Grassmannian if and only if $y\gg N$ is I-Grassmannian for all $N \in \ZZ$.

\begin{corollary}\label{igrass-cor}
If $y \in \I_\ZZ$ is I-Grassmannian and $E \subset \ZZ$ is a finite set with $y(E)=E$,
then the standardized involution $[y]_E$ is also I-Grassmannian.
\end{corollary}

\begin{proof}
The result is evident from Proposition-Definition \ref{propdef-2}(b) and the observation just noted.
\end{proof}

If $y \in \I_\ZZ-\{1\}$ is I-Grassmannian then
$y = (\phi_1, n+1)(\phi_2,n+2)\cdots (\phi_r,n+r)$ for some integers  $r \in \PP$ and
$\phi_1 < \phi_2<\dots < \phi_r \leq n $.
In this case, define the \emph{shape} of  $y$ to be the strict partition
\[ \ilambda(y) = (n+1-\phi_1,n+1-\phi_2,\dots,n+1-\phi_r).\]
Define the shape of $1 \in \I_\ZZ$ to be the empty partition $\ilambda(1)= \emptyset=(0,0,\dots)$.
One can check that if $y \in \I_\infty$ then $\ilambda(y)$ is
the transpose of the partition given by reversing $\hat c(y)$.
Moreover, the map $y \mapsto \ilambda(y)$ restricts to a bijection from $n$-I-Grassmannian involutions to strict
partitions whose parts all have size at most $n$.
Recall the definition of the operators $\pi_{b,a}$  from Section~\ref{divided-sect}.

\begin{lemma}\label{igrass-lem4}
Assume $y \in \I_\infty-\{1\}$ is I-Grassmannian
so that $y = (\phi_1, n+1)(\phi_2,n+2)\cdots (\phi_r,n+r)$ for some
integers
 $1\leq \phi_1 <\phi_2<\dots < \phi_r \leq n$.
Then
$ \iS_y = \pi_{\phi_1,1} \pi_{\phi_2,2} \cdots \pi_{\phi_r,r} \( x^{\ilambda(y)} G_{r,n}\).$
\end{lemma}

\begin{proof}
Let $\Sigma(\phi) = \sum_{i=1}^r (\phi_i-i)$.
If $\Sigma(\phi)=0$, then 
$y= (1,n+1)(2,n+2)\cdots (r,n+r) $ and the lemma asserts that
$\iS_y = x_1^{n} x_2^{n-1}\cdots x_r^{n-r+1} G_{r,n}$, which holds by Theorem \ref{dominant-thm}
since
$\hat D(y) = \{ (i+j,i ): 1\leq i \leq r\text{ and } 0 \leq j \leq n-i\} $.
Suppose $\Sigma(\phi)>0$ and let $i \in [r]$ be the smallest index such that $i<\phi_i$.
It suffices to show that
$ \iS_y = \pi_{\phi_i,i} \pi_{\phi_{i+1},i+1} \cdots \pi_{\phi_r,r} \( x^{\ilambda(y)} G_{r,n}\)$.
Let
\[
v =(1,n+1)(2,n+2)\cdots (i,n+i)(\phi_{i+1},n+i+1)(\phi_{i+2},n+i+2)\cdots(\phi_r,n+r) \in \I_\infty.
\]
Equation \eqref{i-eq} implies  $\iS_y = \partial_{\phi_i,i} \iS_v$,
and
by induction
$ \iS_v = \pi_{\phi_{i+1},i+1} \pi_{\phi_{i+2},i+2}\cdots\pi_{\phi_r,r}\(x^{\ilambda(v)} G_{r,n}\)$.
Since $x^{\ilambda(v)} = x_i^{\phi_i-i} x^{\ilambda(y)}$ and since
multiplication by $x_i$ commutes with $\pi_{\phi_j,j}$ when $i<j$,
we have
\be\label{thiseq}
\iS_y
=\partial_{\phi_i,i} \iS_v
=\partial_{\phi_i,i}\( x_i^{\phi_i-i} \pi_{\phi_{i+1},i+1} \pi_{\phi_{i+2},i+2}\cdots\pi_{\phi_r,r}
\(x^{\ilambda(y)} G_{r,n}\)\).
\ee
Since 
$\partial_j\( x_i^{i-\phi_i} \iS_v\) = x_i^{i-\phi_i} \partial_j \iS_v = 0$
for $i+1\leq j < \phi_i$ as $s_j \notin \DesR(v)$,  the desired identity
$\iS_y=
\pi_{\phi_{i},i}  \pi_{\phi_{i+1},i+1} \cdots \pi_{\phi_r,r} \(x^{\ilambda(y)} G_{r,n}\)
$
follows from \eqref{thiseq}
by  Lemma \ref{pi-lem}.
\end{proof}

\begin{theorem}\label{igrass-thm}
If $y \in \I_\ZZ$ is I-Grassmannian, then
$\iF_y = P_{\ilambda(y)}$.
\end{theorem}

\begin{proof}
Since $\iF_y = \iF_{y \gg N}$ for all $N \in \ZZ$,
 we may assume that 
  $y \in \I_\infty$ is $n$-I-Grassmannian. If $\ilambda(y)$ has $r$ parts, then
 Lemma~\ref{igrass-lem4}
implies
that
$ \pi_{w_n} \iS_y = \pi_{w_n} \(x^{\ilambda(y)} G_{r,n}\)$  for all $n \geq r$,
and the theorem follows by taking the limit  as $n\to \infty$.
\end{proof}

\begin{remark}
It may happen that $\iS_y \neq \rho_n P_{\ilambda(y)}$ when $y \in \I_\infty$
is $n$-I-Grassmannian.
\end{remark}

\subsection{Schur $P$-positivity}\label{ppos-sect}

In this section we describe an algorithm to expand $\iF_y$ 
into a nonnegative linear combination of Schur $P$-functions. Our approach
is inspired by Lascoux and Sch\"utzenberger's original proof of
Theorem~\ref{intro-ls-thm} from \cite{LS}, which we sketch as follows.
Order $\ZZ\times \ZZ$ lexicographically.
Recall the definition of $\Phi^\pm(w,r)$ from Section~\ref{invtrans-sect}.
For $w \in S_\ZZ$, define
$\fk T_1(w)$ to be the empty set if $w$ is Grassmannian, and otherwise
let $\fk T_1(w) = \Phi^-(w(r,s),r)$ where 
 $(r,s)$ is the (lexicographically) maximal element of $\inv(w)$.
One can check that if $(r,s)$ is the maximal inversion of $w \in S_\ZZ$,
then $w(r,s)\lessdot w$ and  $\Phi^+(w(r,s),r) = \{w\}$
and  $r \in \ZZ$ is the largest integer such that $w(r)>w(r+1)$.

\begin{definition}
The \emph{Lascoux-Sch\"utzenberger tree} $\fk T(w)$ of $w \in S_\ZZ$
is the 
tree with root $w$, in which the children of any vertex $v \in S_\ZZ$
are the elements of $\fk T_1(v)$.
\end{definition}

A given permutation
may correspond to more than one vertex in $\fk T(w)$. One can show that  $\fk T(w)$ is always finite \cite{LS}.
Since  $F_w = \sum_{v \in \fk T_1(w)} F_v$ for any non-Grassmannian permutation
 by Theorem \ref{transition-thm2},
 it follows that
 $F_w = \sum_{v } F_v$
where
the sum is over the finite set of  leaf vertices $v$ in $\fk T(w)$.
The leaves of $\fk T(w)$ are Grassmannian permutations by construction,
so Theorem~\ref{intro-ls-thm} follows.

\begin{example}
The Lascoux-Sch\"utzenberger tree $\fk T(w)$ of $w = 1254376 \in S_7$ is shown below.
The maximal inversion of each vertex is underlined.
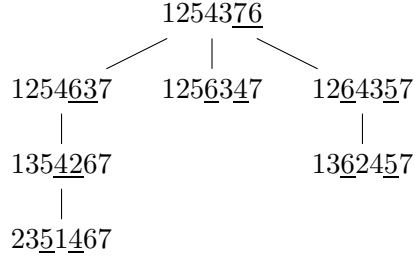
\[
\begin{tikzpicture}
  \node (max) at (0,1) {
12543\underline 7\underline 6
};
  \node (a) at (-2,0) {
1254\underline6\underline37
};
  \node (b) at (0,0) {
125\underline63\underline47
};
  \node (c) at (2,0) {
12\underline 643\underline 57
};
  \node (d) at (-2,-1) {
135\underline4\underline267
};
  \node (e) at (2,-1) {
13\underline 624\underline57
};
  \node (f) at (-2,-2) {
23\underline51\underline467
};
  \draw  [-]
(max)   edge  (a)
(max)   edge  (b)
(max)   edge  (c)
(e)   edge  (c)
(a) edge (d)
(d)   edge  (f)
;
\end{tikzpicture}
\]
It follows from Theorem \ref{classical-grass-thm1} that
$ F_{1254376} =  s_{(3,2,2)} + s_{(3,3,1,1)} + s_{(4,2,1)}$.
\end{example}

Fix $z \in \I_\ZZ$.
Recall that an inversion $(i,j) \in \inv(z)$ is {visible} if $z(j) \leq \min\{ i, z(i)\}$,
and that 
 $i \in \ZZ$ is a {visible descent} of $z$ if $(i,i+1)$
is a visible inversion.
It follows by Lemma \ref{desi-lem} that if $z$
has no visible descents then $\alpha_{\min}(z)=1$ so $z=1$.

\begin{lemma}\label{minvisdes-lem}
Let $z \in \I_\ZZ-\{1\}$ and suppose $j \in \ZZ$ is the smallest integer such that $z(j)<j$.
Then $j-1$ is the minimal visible descent of $z$.
\end{lemma}

\begin{proof}
By hypothesis $z(j) \leq j-1 \leq z(j-1)$ so $j-1$ is a visible descent of $z$, and if $i<j-1$ then
$i+1 \leq z(i+1)$ so $i$ is not a visible descent as $z(i+1) \not \leq i$.
\end{proof}

\begin{lemma}\label{maxvisinv-lem}
Suppose $(q,r) \in \ZZ\times \ZZ$ is the maximal visible inversion of $z \in \I_\ZZ-\{1\}$.
Let $m $ be the largest element of $\supp(z)$.
Then $q$ is the maximal visible descent of $z$ while $r$
is the maximal integer with $z(r) \leq \min\{q,z(q)\}$,
and  we have 
$z(q+1) < z(q+2) < \dots < z(m) \leq q.$
In addition,
 either (a) $z(q) <q <r \leq m$, (b) $z(q)=q<r =m$, or (c) $q<z(q) =r=m$.
\end{lemma}

\begin{proof}
Since $(q+1,r)$ is not a visible inversion of $z$,
we have $z(q+1) \leq \min \{q,z(q)\}$ so $q$ is a visible descent.
If $d$ is another visible descent  of $z$ then $(d,d+1)$ is a visible inversion,
so $d\leq i$. It is clear by definition that $r$ is maximal such that $z(r) \leq \min\{q,z(q)\}$.
We must have $z(q+1) < z(q+2)  <\dots < z(m) \leq q$ since otherwise 
$z$ would have a visible inversion greater than $(q,r)$.
It follows that $r=m$ if $z(q) = q$, and that $z(q)=r=m$ if $q < z(q)$.
\end{proof}

To each nontrivial element of $\I_\ZZ$, we associate a Bruhat covering relation in the following way.

\begin{propdef}\label{ipsi-propdef}
Suppose $(q,r)$ is the maximal visible inversion of $z \in \I_\ZZ-\{1\}$.
There exists a unique involution $\ipsi(z) \in \I_\ZZ$
such that $\ipsi(z) \lessdot z$ and $z=\tau_{qr}(\ipsi(z))$.
The involution $\ipsi(z)$ is as specified in Table \ref{it-fig},
and it holds that $\ipsi(z)(q)\leq q$ and $\ipsi(z)(q)< z(q) \leq \ipsi(z)(r)$.
\end{propdef}

\begin{proof}
If  $y\in \I_\ZZ$ exists such that $y < z$
and $z = \tau_{qr}(y)$, then $y$ is unique and belongs to the set of 
permutations with the same restriction as $z$ to the complement of
$ \{ q, z(q), r, z(r)\}$ in $\ZZ$.
Since $(q,r)$ is the maximal visible inversion of $z$, we have either
$z(r) = q< z(q)=r$ or $z(r)<q=z(q)<r$ or $z(r) < z(q)<q < r$.
Consulting Table \ref{ct-fig},
we deduce that $\ipsi(z)$ exists and is given by the element 
specified in Table \ref{it-fig}. Moreover, we have $\ipsi(z) \lessdot z$ by the previous lemma and Theorem~\ref{taubruhat-thm}.
\end{proof}

\begin{table}[h]
\[
\barr{| c | c | c | c | c | l}
\hline&&&&\\
A=\{q,r,z(q),z(r)\} & [z]_A & (q,r) & [\ipsi(z)]_A & \sigma\text{ such that }\ipsi(z)=z\sigma
\\&&&&\\
\hline
&&&&\\
\{a<b\}
&
\arcstart
{
*{\bullet}   \arc{.6}{r}  & *{\bullet}
}
\arcstop
&(a,b)&
\arcstart
{
*{\bullet}     & *{\bullet}
}
\arcstop
&
(a,b)
\\&&&&\\
\hline
&&&&\\
\{a<b<c\}
&
 \arcstart
{
*{.}    \arc{.8}{rr}  & *{\bullet} & *{\bullet}
}
\arcstop
&
(b,c)
&
 \arcstart
{
*{.}    \arc{.6}{r}  & *{\bullet} & *{\bullet}
}
\arcstop
&
(a,b,c)
\\&&&&\\
\hline&&&&\\
\{a<b<c<d\}
&
\arcstart
{
*{.}  \arc{.8}{rrr}   & *{.}  \arc{.4}{r}  & *{\bullet} & *{\bullet} 
}
\arcstop
& (c,d)
&
  \arcstart
{
*{.}  \arc{.8}{rr}   & *{.} \arc{.8}{rr}   & *{\bullet} & *{\bullet} 
}
\arcstop
&
(a,b)(c,d)
\\
&&&&\\\hline
\earr
\]
\caption{Values of $\ipsi(z)$.
Fix $z \in \I_\ZZ$ with maximal visible inversion $(q,r)$.
 Let $A = \{q,r,z(q),z(r)\}$.
The first column labels the elements of $A$.
The third column rewrites $(q,r)$ in this labeling. The last two columns
determine $\ipsi(z)$ as characterized in Proposition-Definition \ref{ipsi-propdef}.
In the second and fourth columns,
we use $\bullet$ symbols to mark the vertices corresponding to $q$ and $r$.
}\label{it-fig}
\end{table}

As $\ipsi(z)$ is only defined if $z$ has a visible inversion, we
view $\ipsi$ as a map $\I_\ZZ-\{1\} \to \I_\ZZ$.

\begin{remark}\label{eta-rmk}
Suppose $z \in \I_\ZZ-\{1\}$
has maximal visible inversion $(q,r)$.
Let $p = z(r)$, $y = \ipsi(z)$, and $m = \max \supp(z)$.
Lemma~\ref{maxvisinv-lem} completely
 determines the values of $y(i)$ and $z(i)$ for all $i \geq q$,
 and there are three qualitatively distinct cases for what can happen.

\ben
\item[(a)]  If $z(q) <q <r \leq m$ then $y = (q,r)z(q,r)$ and $z$ correspond to the  pictures
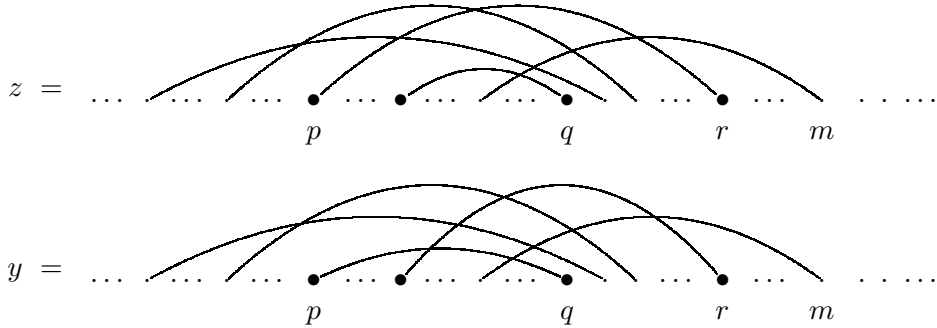
\[
\ba
\\
\\
z &\ =\  
\arcstartc{.2}
{
*{\dots}&
*{.} \arc{2}{rrrrrrrrrrr}&
*{\dots}&
*{.} \arc{3}{rrrrrrrrrr}&
*{\dots}&
*{\bullet} \arc{3}{rrrrrrrrrr} &
*{\dots} &
*{\bullet} \arc{1}{rrrr} &
*{\dots} &
*{.} \arc{2}{rrrrrrrr}&
*{\dots} &
*{\bullet} &
*{.} &
*{.} &
*{\dots} &
*{\bullet} &
*{\dots} &
*{.} &
*{.} &
*{.} &
*{\dots}
\\
&
&
&
&
&
p
&
&
&
&
&
&
q
&
&
&
&
r
&
&
m
}
\arcstop
\\
\\
\\
y&\ =\  
\arcstartc{.2}
{
*{\dots}&
*{.} \arc{2}{rrrrrrrrrrr}&
*{\dots}&
*{.} \arc{3}{rrrrrrrrrr}&
*{\dots}&
*{\bullet} \arc{1}{rrrrrr} &
*{\dots} &
*{\bullet} \arc{3}{rrrrrrrr} &
*{\dots} &
*{.} \arc{2}{rrrrrrrr}&
*{\dots} &
*{\bullet} &
*{.} &
*{.} &
*{\dots} &
*{\bullet} &
*{\dots} &
*{.} &
*{.} &
*{.} &
*{\dots}
\\
&
&
&
&
&
p
&
&
&
&
&
&
q
&
&
&
&
r
&
&
m
}
\arcstop
\ea
\]
In our diagrams of this kind, each ellipsis ``$\dots$'' stands for zero or more unspecified vertices.
 Lemma~\ref{maxvisinv-lem} implies that
 $z(q+1) < z(q+2) < \dots <z(r) < z(q)$, and that if $r<m$ then $z(q)<z(r+1) < z(r+2) < \dots < z(m)  < q$.

\item[(b)]  If $z(q) = q < r=m$ then  $y = (q,r)z(q,r)$ and $z$ may be represented as 
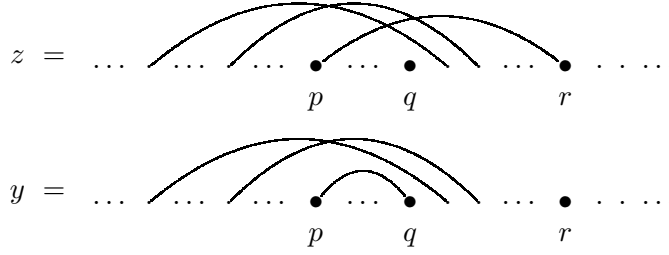
\[
\ba
\\
z &\ =\  
\arcstartc{.2}
{
*{\dots}&
*{.} \arc{2}{rrrrrrr}&
*{\dots}&
*{.} \arc{2}{rrrrrr}&
*{\dots}&
*{\bullet} \arc{1.6}{rrrrrr} &
*{\dots} &
*{\bullet} &
*{.} &
*{.} &
*{\dots} &
*{\bullet} &
*{.} &
*{.} &
*{\dots}
\\
&
&
&
&
&
p
&
&
q
&
&
&
&
r
}
\arcstop
\\
\\
y &\ =\  
\arcstartc{.2}
{
*{\dots}&
*{.} \arc{2}{rrrrrrr}&
*{\dots}&
*{.} \arc{2}{rrrrrr}&
*{\dots}&
*{\bullet} \arc{1}{rr} &
*{\dots} &
*{\bullet} &
*{.} &
*{.} &
*{\dots} &
*{\bullet} &
*{.} &
*{.} &
*{\dots}
\\
&
&
&
&
&
p
&
&
q
&
&
&
&
r
}
\arcstop
\ea
\]
In this case,   $z(q+1) < z(q+2) < \dots <z(r) < q$, so $z(i) < q$ if $p<i<q$.

\item[(c)] If $q<z(q) = r=m$ so that $p=q$, then $y = z(q,r)$ and $z$ may be represented as
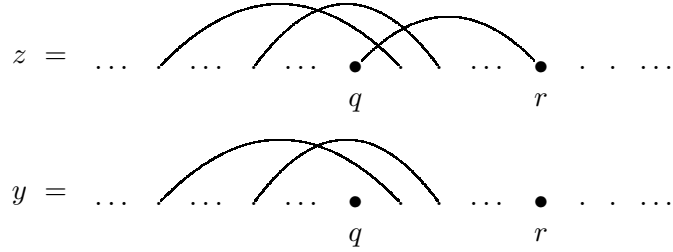
\[
\ba
\\
z &\ =\  
\arcstartc{.3}
{
*{\dots}&
*{.} \arc{2}{rrrrr}&
*{\dots}&
*{.} \arc{2}{rrrr}&
*{\dots}&
*{\bullet} \arc{1.6}{rrrr} &
*{.} &
*{.} &
*{\dots} &
*{\bullet} &
*{.} &
*{.} &
*{\dots}
\\
&
&
&
&
&
q
&
&
&
&
r
}
\arcstop
\\
\\
y &\ =\  
\arcstartc{.3}
{
*{\dots}&
*{.} \arc{2}{rrrrr}&
*{\dots}&
*{.} \arc{2}{rrrr}&
*{\dots}&
*{\bullet} &
*{.} &
*{.} &
*{\dots} &
*{\bullet} &
*{.} &
*{.} &
*{\dots}
\\
&
&
&
&
&
q
&
&
&
&
r
}
\arcstop
\ea
\]
In this case  $z(q+1) < z(q+2) < \dots <z(r-1) < q$.
\een
\end{remark}

\begin{lemma} \label{it-atom-lem}
If $(q,r)$ is the maximal visible inversion of $z \in \I_\infty-\{1\}$ and
$w = \alpha_{\min}(z)$ is the minimal atom of $z$,
then $w(q,r) = \alpha_{\min}(\ipsi(z))$ is the minimal atom of $\ipsi(z)$.
\end{lemma}

\begin{proof}
Let $\Cyc_\PP(z) = \{ (a_i,b_i) : i \in \PP\}$ and $\Cyc_\PP(\ipsi(z)) = \{ (c_i,d_i): i \in \PP\}$
where $a_1<a_2<\dots$ and $c_1<c_2<\dots$.
By Lemma \ref{minatom-lem}, it suffices to show that interchanging $q$ and $r$ and removing
all repeated letters after their first appearance in 
$b_1a_1b_2a_2\cdots$ gives the same word as removing the repeated letters in $d_1c_1d_2c_2\cdots$.
This is   straightforward  from  Remark~\ref{eta-rmk}.
For example,
if $p=z(r) <  q = z(q) < r$, then for some $n \in \PP$
we have $b_na_nb_{n+1}a_{n+1} = rpqq$, $d_nc_nd_{n+1}c_{n+1} = qprr$,
 and $(a_i,b_i)=(c_i,d_i)$ for all $i\neq n$, in which case the desired property is clear.
\end{proof}

Recall the definition of the sets $\hat \Phi^+(y,r)$ and $\hat \Phi^-(y,r)$
from Section \ref{invtrans-sect}.

\begin{lemma}\label{it-lem00}
If  $z \in \I_\ZZ-\{1\}$ has  maximal visible inversion $(q,r)$ then
 $\hat \Phi^+(\ipsi(z),q)= \{z\}$.
\end{lemma}

\begin{proof}
This follows  from Theorem~\ref{taubruhat-thm},
Remark~\ref{eta-rmk}, and the definitions of $\ipsi(z)$ and $\hat \Phi^+(y,q)$.
\end{proof}

We may now define an involution analogue of the set $\fk T_1(w)$.
For $z \in \I_\ZZ$,  let
\[   \iT_1(z) = \begin{cases}
\varnothing &\text{if  $z$ is I-Grassmannian} \\
 \hat\Phi^-(y, p) &\text{otherwise}
 \end{cases}
 \]
 where in the second case, we set $y=\ipsi(z)$ and $p =y(q)$
 with  $q$ the maximal visible descent of $z$.
 If $z$ is not I-Grassmannian
 then $\iT_1(z)\neq \varnothing$.

\begin{definition}\label{itree-def}
The \emph{involution Lascoux-Sch\"utzenberger tree} $\iT(z)$ of $z \in \I_\ZZ$
is the
tree with root $z$, in which the children of any vertex $v \in \I_\ZZ$
are the elements of $\iT_1(v)$.
\end{definition}

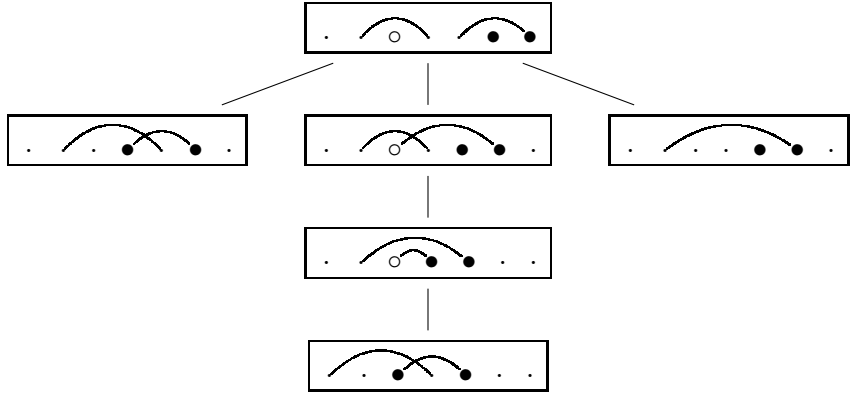
\begin{figure}[h]
\[
\begin{tikzpicture}
  \node (max) at (0,1.5) {
\boxed{\arcstart
{
  { }
  \\
  *{.} &
*{.}  \arc{.6}{rr} &
*{\circ} &
*{.} &
*{.}  \arc{.6}{rr} & 
*{\bullet} &
*{\bullet} 
}
\arcstop}
};
  \node (a) at (-4,0) {
\boxed{\arcstart
{
  { }
  \\
  *{.} &
*{.}  \arc{.8}{rrr} &
*{.} &
*{\bullet}  \arc{.6}{rr} &
*{.}  & 
*{\bullet} &
*{.} 
}
\arcstop}
};
  \node (b) at (0,0) {
\boxed{\arcstart
{
  { }
  \\
  *{.} &
*{.}  \arc{.6}{rr} &
*{\circ}  \arc{.8}{rrr}&
*{.}  &
*{\bullet}  & 
*{\bullet} &
*{.} 
}
\arcstop}
};
  \node (c) at (4,0) {
\boxed{\arcstart
{
  { }
  \\
  *{.} &
*{.}  \arc{.8}{rrrr} &
*{.} &
*{.} &
*{\bullet}  & 
*{\bullet} &
*{.} 
}
\arcstop}
};
  \node (d) at (0,-1.5) {
\boxed{\arcstart
{
  { }
  \\
  *{.} &
*{.}  \arc{.8}{rrr} &
*{\circ}  \arc{.4}{r}&
*{\bullet}  &
*{\bullet}  & 
*{.} &
*{.} 
}
\arcstop}
};
  \node (e) at (0,-3) {
\boxed{\arcstart
{
  { }
  \\
*{.}  \arc{.8}{rrr} &
*{.}  &
*{\bullet}  \arc{.6}{rr}&
*{.}  & 
*{\bullet} &
*{.} &
*{.}
}
\arcstop}
};
  \draw  [-]
(max)   edge  (a)
(max)   edge  (b)
(max)   edge  (c)
(e)   edge  (d)
(d)   edge  (b)
;
\end{tikzpicture}
\]
\caption{The involution Lascoux-Sch\"utzenberger tree $\iT(z)$ for $z = (2,4)(5,7) \in \I_7$.
The maximal visible inversion of each vertex is marked with $\bullet$,
and the minimal visible descent in each non-leaf is marked with $\circ$.
From Theorem \ref{igrass-thm} and Corollary \ref{it-cor1},
one computes that $\iF_{z} = 2 P_{(3,1)} + P_{(4)}$.
}
\label{itree-fig}
\end{figure}

As with $\fk T(w)$,  an involution is allowed to correspond to more than one vertex in $\iT(z)$.
All vertices $v$ in $\iT(z)$ satisfy $\ellhat(v) = \ellhat(z)$ by construction,
 so $1$ is not a vertex unless $z=1$.
An example tree $\iT(z)$ is shown in Figure~\ref{itree-fig}.
Recall that $x_{(p,q)}$ is $x_p+x_q$ if $p\neq q$ and $x_p=x_q$ otherwise. 

\begin{corollary}\label{it-cor1}
Suppose $z \in \I_\ZZ$ is an involution which is not I-Grassmannian,
whose maximal visible descent is $q \in \ZZ$. The following identities then hold:
\ben
\item[(a)] $\iS_z =x_{(p,q)} \iS_{y} +  \sum_{v \in \iT_1(z)} \iS_v$ where $y = \ipsi(z)$ and $p = y(q)$.

\item[(b)] $\iF_z = \sum_{v \in \iT_1(z)} \iF_v$.
\een
\end{corollary}

\begin{proof}
The result is immediate from Theorems \ref{invmonk-thm} and \ref{invmonk-thm2} and Lemma \ref{it-lem00}.
\end{proof}

To show that $\iT(z)$ is a finite tree, we depend on a sequence of technical lemmas.
 Note that $\ipsi(z) =1$ if and only if $z$ is a transposition, in which case $z$ is I-Grassmannian.

\begin{lemma}\label{it-lem0}
Suppose $z \in \I_\ZZ$ is not I-Grassmannian, so that $\ipsi(z)\neq 1$.
\ben
\item[(a)] The maximal visible descent of $\ipsi(z)$ is less than or equal to that of $z$.
\item[(b)] The minimal visible descent of $\ipsi(z)$ is equal to that of $z$.
\een
\end{lemma}

\begin{proof}
We may assume that $z \in \I_\infty$.
In view of Proposition-Definition \ref{propdef-2} and Lemmas \ref{desi-lem0} and \ref{it-atom-lem},
it suffices to show that if $(i,j)$ is the maximal inversion of a permutation $w \in S_\infty$
which is not Grassmannian, then the maximal (respectively, minimal) descent of $w(i,j)$
is at most (respectively, equal to) that of $w$.
This is a straightforward exercise which is left to the reader.
%
  \end{proof}

\begin{lemma}\label{it-tau-lem}
If $y \in \I_\ZZ$ and $n < p \leq q = y(p)<r$ then $\tau_{np}(y)(q) \leq y(q)$ and
$\tau_{np}(y)(r) = y(r)$.
\end{lemma}

\begin{proof}
The result follows from the definition of $\tau_{np}$; see Table \ref{ct-fig}.
\end{proof}

\begin{lemma}\label{it-lem1}
Suppose $z \in \I_\ZZ$ is not I-Grassmannian.
Let $i$ and $j$ be the minimal and maximal visible descents of $z$, and suppose $v \in \iT_1(z)$.
If $d$ is a visible descent of $v$, then $i\leq d \leq j$.
\end{lemma}

\begin{proof}
Let $p=\ipsi(z)(j) \leq j$, let $d$ be a visible descent of $v$, and let $n<p$ be such that
 $v = \tau_{np}(\ipsi(z))$.
By Lemma \ref{it-tau-lem}, we have $v(k)=\ipsi(z)(k)$ for all $k>j$.
As the maximal visible descent of $\ipsi(z)$ is at most $j$ by Lemma \ref{it-lem0}(a), 
we deduce that $d \leq j$.

Define $a,b \in \ZZ$ as the smallest integers such that
$\ipsi(z)(a)<a$ and $v(b)<b$. It follows from Lemmas \ref{minvisdes-lem} and \ref{it-lem0}(b)
that $i=a-1$ and that $b-1$ is the minimal visible descent of $v$,
so to prove that $i \leq d$ it suffices to show that $a\leq b$.
This is clear from the definition of $\tau_{np}$
except when $n$ and $p$ are both fixed points of $\ipsi(z)$,
in which case it could occur that $b=p$.
In this situation, however,
we would have $p=\ipsi(z)(j)=j$,
 so $a\leq b$ would hold anyway since $a=i+1\leq j$.
\end{proof}

For any $z \in \I_\ZZ$, let $\iT_0(z) = \{ z\}$
and define $\iT_n(z) = \bigcup_{v \in \iT_{n-1}(z)} \iT_1(v)$ for $n\geq 1$.

\begin{lemma}\label{it-lem2}
Suppose $z \in \I_\ZZ$ and $v \in \iT_1(z)$.
Let $(q,r)$ be the maximal visible inversion of $z$,
and let $(q_1,r_1)$ be any visible inversion of $v$. Then $q_1<q$ or $r_1<r$.
Hence, if $n\geq r-q$ then the maximal visible descent of every element of $\iT_n(z)$ is
strictly less than $q$.
\end{lemma}

\begin{proof}
Lemmas \ref{maxvisinv-lem} and \ref{it-lem1} imply that $q_1\leq q$, so suppose $q_1=q$.
Since $v(q) \leq \ipsi(z)(q) < \ipsi(z)(r) =v(r) $
by Proposition-Definition \ref{ipsi-propdef} and Lemma \ref{it-tau-lem}, $(q,r)$
is not a visible inversion of $v$.
If $s>r$, then $z(q) < z(s)$ by
Lemma \ref{maxvisinv-lem}, while $v(q) \leq \ipsi(z)(q) \leq z(q)$ and $v(s) = \ipsi(z)(s) = z(s)$
by Lemma \ref{it-tau-lem} and the definition of $\ipsi(z)$,
so $(q,s)$ is also not a visible inversion of $v$.
Thus $r_1<r$.
 \end{proof}

 \begin{theorem}\label{it0-thm}
 The involution Lascoux-Sch\"utzenberger tree $\iT(z)$ is finite for all $z \in \I_\ZZ$,
 and  it holds
 that $\iF_z = \sum_v \iF_v$ where the sum
  is over the finite set of leaf vertices $v$ in $\iT(z)$.
 \end{theorem}

\begin{proof}
It follows by induction from Lemmas \ref{it-lem1} and \ref{it-lem2}
that for some sufficiently large $n$
either $\iT_n(z) = \varnothing$ or all elements of $\iT_n(z)$ are I-Grassmannian,
and in the latter case $\iT_{n+1}(z) = \varnothing$.
The tree $\iT(z)$ is therefore finite,  so the identity $\iF_z = \sum_v \iF_v$ follows from
Corollary \ref{it-cor1}.
\end{proof}

The theorem
implies this corollary, which we stated in the introduction as Theorem~\ref{our-thm2}.

  \begin{corollary}\label{it-thm-cor0}
If $z \in \I_\ZZ$ then
$\iF_z \in \NN\spanning\left\{ \iF_y : y \in \I_\ZZ\text{ is I-Grassmannian}\right\}$ and this symmetric function
is consequently Schur $P$-positive.
 \end{corollary}

\subsection{Triangularity}\label{tri-sect}

Recall the definitions of $c(w)$ for $w \in S_\infty$ and $\ic(y)$ for $y \in \I_\infty$
from Section~\ref{igrass-sect}.
The \emph{shape} of $w \in S_\infty$ is the partition $\lambda(w)$ given by sorting $c(w)$.
For involutions, we have this alternative:

\begin{definition}\label{ishape-def}
Let $\ilambda(y)$ for
 $y \in \I_\infty$ be the transpose of the partition
 given by sorting $\ic(y)$.
\end{definition}

These constructions are consistent with our  definitions of $\lambda(w)$ and $\ilambda(y)$
when $w$ is Grassmannian and $y$ is I-Grassmannian.
Let $<$ be the dominance order on partitions,
and write $\mu^T$ for the transpose of a partition $\mu$.
Recall that $\lambda \leq \mu$ if and only if $\mu^T \leq \lambda^T$ \cite[Eq.\ (1.11), {\S}I.1]{Macdonald2}.

\begin{theorem}[Stanley \cite{Stan}] \label{ut-thm1}
Let $w \in S_\infty$ and define $\lambda'(w) =\lambda(w^{-1})^T$.
Then $\lambda(w) \leq \lambda'(w)$,
and if equality holds then $F_w = s_{\lambda(w)}$ while otherwise
$
F_w \in s_{\lambda(w)} + s_{\lambda'(w)} +
\NN\spanning \left\{ s_\nu : \lambda(w) < \nu < \lambda'(w)\right\}.
$
\end{theorem}

Stanley \cite[Theorem 4.10]{Stan} only established the form of this expansion;
the positivity of its coefficients follows from
 results of Edelman and Greene \cite{EG}.
In this section, we prove an analogous result for the decomposition of $\iF_y$ into Schur $P$-functions.

Define $<_\cA$ on $S_\infty$
as the transitive relation generated by setting $v <_\cA w$
when the one-line representation of $v^{-1}$ can be transformed to that of $w^{-1}$ by replacing a consecutive subsequence 
of the form $cab$ with $a<b<c$ by $bca$,
or equivalently when $v < s_{i+1} v =s_{i}w> w$ for some $i \in \PP$.
For example, $ 3412 =(3412)^{-1}  <_\cA  (3241)^{-1} =4213$.
 Recall the definition of $\alpha_{\min}(y)$
 from Lemma~\ref{minatom-lem}.
In prior work, we showed
\cite[Theorem 6.10]{HMP2} that $<_\cA$ is a partial order and
that $\cA(y) = \{ w \in S_\infty : \alpha_{\min}(y) \leq_\cA w\}$ for all $y \in \I_\infty$.

\begin{lemma}\label{ut-lem1}
Let $y \in \I_\infty$. If $v,w \in \cA(y)$ and $v <_\cA w$, then
$\lambda(v) < \lambda(w)$.
\end{lemma}

\begin{proof}
Fix $v,w \in \cA(y)$ with $v <_\cA w$. It suffices to consider the case when
$w$ covers $v$, so assume $v < s_{i+1} v =s_{i}w> w$  for some $i \in \PP$.
Let $a = w^{-1}(i+2)$, $b=w^{-1}(i)$, and $c=w^{-1}(i+1)$, so that $a<b<c$.
If $u \in S_\infty$ and $u < us_j$ for some $j \in \PP$, then
 the diagram $D(us_j)$ is given
 by transposing rows $j$ and $j+1$ of the union $D(u)\cup\{(j+1,u(j))\}$.
 It follows that $D(v^{-1})$
 is given by permuting rows $i$, $i+1$, and $i+2$ of $D(w^{-1}) \cup \{(i+1,b)\} - \{(i,a)\}$.
There are evidently at least two more positions in column $a$ than $b$ of $D(w^{-1})$,
 so as $D(u^{-1}) = D(u)^T$ for any $u \in S_\infty$,
we deduce that $\lambda(v)=\lambda(w) - e_j + e_k$ for some indices $j<k$,
and hence that $\lambda(v) < \lambda(w)$.
\end{proof}

\begin{theorem}\label{ut-thm2}
Let $y \in \I_\infty$ and $\mu = \ilambda(y)$.
Then $\mu^T \leq \mu$. If $\mu^T = \mu$ then
$\iF_y = s_\mu$ and otherwise
$\iF_y \in s_{\mu^T} + s_{\mu} +
\NN\spanning \left\{ s_\lambda : \mu^T<\lambda < \mu\right\}$.
\end{theorem}

\begin{proof}
Since $\iF_y = \sum_{w \in \cA(y)} F_w$,
 Theorem~\ref{ut-thm1} and Lemma~\ref{ut-lem1} imply
that
 $\iF_y \in s_\nu  + \NN\spanning \{ s_\lambda : \nu < \lambda \}$
 for $\nu = \lambda(\alpha_{\min}(y))$.
 Write $\omega : \Lambda \to \Lambda$ for the linear map with $\omega(s_\lambda) = s_{\lambda^T}$
 for partitions $\lambda$.
 If $\lambda $ is strict then $\omega(P_\lambda) = P_\lambda$
 \cite[Example 3(a), {\S}III.8]{Macdonald2}, so $\omega (\iF_y) = \iF_y$
 by Corollary~\ref{our-cor1}.
The result follows since 
the definition of $\ilambda(y)$ and Lemma~\ref{codei-lem} imply that
$\mu^T = \nu$.
\end{proof}

The following is equivalent to Theorem~\ref{intro-tri-thm} in the introduction.

\begin{corollary}\label{ut-cor}
If $y \in \I_\infty$ then $\ilambda(y)$ is strict and
$\iF_y \in P_{\ilambda(y)} + \NN\spanning \left\{ P_\lambda : \lambda < \ilambda(y)\right\}$.
\end{corollary}

\begin{proof}
Since $P_\lambda \in s_\lambda + \NN\spanning\{ s_\nu  : \nu < \lambda\}$
for any strict partition $\lambda$ \cite[Eq.\ (8.17)(ii), {\S}III.8]{Macdonald2} and since 
$\iF_y$ is Schur $P$-positive,
the result holds by Theorem~\ref{ut-thm2}.
\end{proof}

\begin{remark}
This is the easiest way we know of showing that $\ilambda(y)$
is a strict partition. There should exist a more direct, combinatorial proof
of this fact,
using just the definition of $\ilambda(y)$.
\end{remark}

We mention some applications  to skew Schur functions.
As is standard,
we write  $\mu \subset \lambda$ and say that
  $\lambda$ \emph{contains} $\mu$
   if $\lambda$ and $\mu$ are partitions
 with $\mu_i \leq \lambda_i$ for all $i \in \PP$.
When $\mu\subset \lambda$, we let $\lambda \setminus \mu$ and $s_{\lambda\setminus \mu}$
 denote the corresponding
 skew shape and skew Schur function.
 We say that $\lambda$ \emph{strictly contains} $\mu$ if
 $0= \mu_i = \lambda_i$ or $0\leq \mu_i < \lambda_i$ for each $i \in \PP$.
 For a partition $\mu$ which is strictly contained in $\delta_{n+1} = (n,n-1,\dots,2,1)$,
we  define
 \be
 \label{ymu-eq}
y_{\mu,n} = (a_1,b_1)(a_2,b_2)\cdots (a_n,b_n) \in \I_{2n}
\ee
 where $b_i = n+i-\mu^T_i$ for $i \in [n]$ and $a_1<a_2<\dots<a_n$ are the numbers in
 $[2n]\setminus \{b_1,b_2,\dots,b_n\}$
 labeled in increasing order.
 Note that $b_1<b_2<\dots<b_n=2n$, and that
 $\mu^T_i < n+1-i$ and $2i -1< b_i $ for each $i \in [n]$.
Thus $a_i < b_i$ for each $i$, so $y_{\mu,n}$ is well-defined and fixed-point-free.

 \begin{example}
 If $\mu = (5,3,2,2) \subset \delta_7$, then $\mu^T = (4,4,2,1,1)$
 and $y_{\mu,6} = (a_1,b_1)(a_2,b_2)\cdots (a_6,b_6)$
for $(b_1,b_2,\dots,b_6) = (3, 4, 7, 9, 10, 12)$
 and $(a_1,a_2,\dots,a_6) = (1, 2, 5, 6, 8, 11)$.
 \end{example}

Two subsets of $\PP\times \PP$ are \emph{equivalent} if one can be transformed to the other
by permuting its rows and columns.
Equivalent skew shapes index equal
skew Schur functions \cite[Proposition 2.4]{BJS}.

 \begin{proposition}\label{ardila-prop1}
Let $n \in \NN$,
 suppose $\mu$ is a partition strictly contained in $\delta_{n+1}$,
 and set $y=y_{\mu,n}$. 
 Then $y$ is  321-avoiding,
  the sets $\hat D(y)$
and $\delta_{n+1}\setminus \mu$ are equivalent,
and $\iF_{y} = s_{\delta_{n+1}\setminus \mu}$.
 \end{proposition}

 \begin{proof}
 It is evident that
$y=y_{\mu,n}$ is 321-avoiding since $a_i<b_i$ and $a_i<a_{i+1}$ and $b_i<b_{i+1}$ for all $i$.
For the same reason, we have
$(i,j) \in \hat D(y)$ only if
$\{i,j\} \subset A$ where $A =\{a_1,a_2,\dots,a_n\}$,
and the positions in column $a_i$ of $\hat D(y)$ are the pairs
$(a_j,a_i)$ with $i \leq j$ and $a_j < b_i$.
Since exactly $n+1-i-\mu_i^T$ elements $a \in A$ satisfy $a_i \leq a < b_i$,
 we deduce that the map
$(a_j,a_i) \mapsto (n+1-j,i)$ is a bijection $\hat D(y) \to \delta_{n+1}\setminus \mu$.
It follows that  $\hat D(y)$
and $\delta_{n+1}\setminus \mu$ are equivalent subsets of $\PP\times \PP$,
so $\iF_{y} = s_{\delta_{n+1}\setminus \mu}$
by
\cite[Proposition 3.31]{HMP1} and the discussion in \cite[Section 2]{BJS}.
 \end{proof}

\begin{lemma}\label{ardila-lem2}
Let $m \in \PP$ and suppose
$\mu\subset \delta_{m}$ is a partition with $\mu \neq \delta_m$.
There exists $n \in \PP$ and a partition $\nu$ strictly contained in $\delta_n$
such that $\delta_m\setminus \mu$ and $\delta_n\setminus \nu$ are equivalent shapes.
\end{lemma}

\begin{proof}
If $\mu_i = m-i$ for some $i \in [m-1]$ then $\delta_m\setminus \mu$ is equivalent to
 $\delta_{m-1}\setminus \nu$ for $\nu = (\mu_1-1,\dots,\mu_{i-1}-1,\mu_{i+1},\dots,\mu_{m-1})$.
 The lemma follows by repeatedly applying  this observation.
\end{proof}

\begin{proposition}
For each $n \in \PP$ and partition $\mu\subset  \delta_n$, there exists $y \in \I_\ZZ$ with
$s_{\delta_n\setminus \mu}=\iF_y$.
\end{proposition}

 \begin{proof}
Since $s_{\varnothing} = \iF_1 = 1$,
it suffices
by Lemma~\ref{ardila-lem2} to prove that $ s_{\delta_n\setminus \mu}=\iF_y$ for some $y \in \I_\ZZ$
when $\mu$ is strictly contained in $\delta_{n+1}$.
This holds for $y=y_{\mu,n}$
by Proposition~\ref{ardila-prop1}.
 \end{proof}

 For a finite set $D\subset \PP\times \PP$, let $\gamma(D)$ be the transpose
 of the partition given by sorting the numbers of positions in each row of $D$.
 For example, if $\mu = (3,3,1)$ then $\gamma(\delta_6\setminus \mu) = (2,2,2,1,1)^T=(5,3)$.
If $\mu \subset \delta_n$ then $\gamma(\delta_n\setminus \mu)$
is the same as what DeWitt calls the \emph{$n$-complement} of $\mu$ \cite[Definition IV.11]{DeWitt},
and is always a strict partition, since its parts count the positions of $\delta_n\setminus \mu$
on each
southwest-to-northeast diagonal.
The following is a weaker version of \cite[Theorem V.5]{DeWitt},
and also closely related to the main result of Ardila and Serrano's paper \cite{AS}.

  \begin{corollary}[DeWitt \cite{DeWitt}]
  \label{ardila-cor1}
  If $\mu\subset \delta_n$ and $\gamma = \gamma(\delta_n\setminus \mu)$, 
  then $s_{\delta_n\setminus \mu} \in P_\gamma + \NN\spanning\{ P_\nu : \nu < \gamma\}$.
\end{corollary}

\begin{proof}
Both $\gamma(D)$ and $s_D$ (when $D$ is a skew shape) are invariant under
equivalences between subsets of $\PP\times \PP$,
so this follows from
Corollary~\ref{ut-cor},
Proposition~\ref{ardila-prop1}, 
and Lemma~\ref{ardila-lem2}.
\end{proof}

\subsection{Vexillary involutions}\label{i-vex-sect}

By \cite[Theorem 3.36]{HMP1}, the involutions $z \in \I_\ZZ$ for which $\iF_z$ is a
Schur function are precisely those
which are Grassmannian in the ordinary sense of having at most one right descent.
This condition is quite restrictive, as $z \in \I_\ZZ$ is Grassmannian if and only if
$z$ is I-Grassmannian with shape $\ilambda(z) = \delta_k = (k-1,\dots,2,1,0)$ for some
$k \in \PP$ \cite[Proposition 3.34]{HMP1}.
In this section we consider the more general problem of classifying the involutions
$z \in \I_\ZZ$ for which $\iF_z = P_\mu$ for some strict partition $\mu$.
As in the introduction, we refer to involutions with this property as \emph{$P$-vexillary}.

\begin{remark}
All I-Grassmannian involutions are $P$-vexillary by Theorem \ref{igrass-thm}.
The sequence $(v_n)_{n\geq 1} = (1, 2, 4, 10, 24, 63, 159, 423, 1099, 2962,7868,\dots)$,
with $v_n$ counting the  $P$-vexillary elements of $\I_n$,
does not appear to be related to any existing entry in \cite{OEIS}.
\end{remark}

Recall that if $E\subset \ZZ$ is a finite set of size $n$ then $\psi_E $
is the unique order-preserving bijection
$ E \to [n]$.
In the next three lemmas, we maintain the following notation:
let  $z \in \I_\ZZ-\{1\}$ be a nontrivial involution with maximal visible inversion $(q,r)$,
set $y = \ipsi(z)$, and write $p=y(q)$
so
that $\iT_1(z) = \hat \Phi^-(y,p)$ if $z$ is not I-Grassmannian.
Recall that $p\leq q$ by
Proposition-Definition \ref{ipsi-propdef}.

\begin{lemma} \label{vex-lem1}
Let $E\subset \ZZ$
be a finite set with $\{q,r\} \subset E$ and $z(E) = E$. 
Then $(\psi_E(q), \psi_E(r))$ is the maximal visible inversion of $[z]_E$
and  it holds that $[\ipsi(z)]_E = \ipsi([z]_E)$.
\end{lemma}

\begin{proof}
The first assertion holds since
the set of visible inversions of $z$ contained in $E\times E$
and the set of all visible inversions of $[z]_E$
are in bijection via the map $\psi_E \times \psi_E$, which preserves lexicographic order.
Since $\{ q,r, z(q), z(r)\} \subset E$,
we have $[\ipsi(z)]_E = \ipsi([z]_E)$
 by the definition of $\ipsi$.
\end{proof}

Write $L(z)$ for the set of integers $i<p$ with $\tau_{ip}(y) \in \hat \Phi^-(y,p)$
and, given a set $E\subset \ZZ$, define
$
\fk C(z, E) = \{  \tau_{ip}(y) : i \in E\cap L(z)\}.
$
Also let $\fk C(z) = \fk C(z, \ZZ)$.
Note that $\fk C(z) = \iT_1(z)$ if $z$ is not I-Grassmannian.
The following shows that $\fk C(z)$ is always nonempty:

\begin{lemma} \label{vex-grass-lem}
If $z \in \I_\ZZ-\{1\}$ is I-Grassmannian, then $|\fk C(z)| = 1$ and $\kappa(v) = \kappa(z)$ if $\fk C(z) = \{v\}$.
\end{lemma}
 
Recall that $\kappa(y)$ is the number of nontrivial cycles of $y \in \I_\ZZ$.
 
\begin{proof}
Suppose $z \in\I_\ZZ-\{1\}$ is I-Grassmannian, so that
$z = (\phi_1, n+1)(\phi_2,n+2)\cdots (\phi_k,n+k)$ for some integers  $k \in \PP$ and
$ \phi_1 < \phi_2<\dots < \phi_k \leq n $
by Proposition-Definition \ref{propdef-2}.
Then $(q,r) = (n, n+k)$ and $z(r)=\phi_k$,
and
if $i \in \ZZ$ is maximal such that $i=z(i)<\phi_k$, then  $\fk C(z) = \{v\}$
where $v=(\phi_1, n+1)(\phi_2,n+2)\cdots (\phi_{k-1},n+k-1)(i,n)$.
\end{proof}

\begin{lemma} \label{vex-lem2}
Let $E\subset \ZZ$ be a finite set
such that $\{q,r\} \subset E$ and $z(E) = E$.
\ben
\item[(a)] The operation $v \mapsto [v]_E$ restricts to an injective map $\fk C(z,E) \to \fk C([z]_E)$.

\item[(b)] If $E$ contains $L(z)$, then the injective map in  (a) is a bijection.
\een
\end{lemma}

\begin{proof}
Part (a) is straightforward from
 Lemma~\ref{vex-lem1} and the definitions of  $\tau_{ip}$ and $\ipsi$ and $E$.
 We prove the contrapositive of part (b). Suppose $a< b=\psi_E(p)$ and
$\tau_{ab}([y]_E) \in \fk C([z]_E)$
but 
$\tau_{ab}([y]_E)$ is not in the image of $\fk C(z,E)$ under the map $v \mapsto [v]_E$.
Let $i \in E$ be such that $\psi_E(i) = a$.
We have $\tau_{ab}([y]_E) = [\tau_{ip}(y)]_E$, and by
Theorem~\ref{taubruhat-thm} it holds that $[y]_E\lessdot [y]_E(a,b)$ and therefore
$[y]_E(a) < [y]_E(b)$ and $y(i) < y(p)$.
Since $y \lessdot y(i,p)$ would imply that $\tau_{ip}(y)\in \fk C(z,E)$
by Theorem~\ref{taubruhat-thm},
there must exist an integer $j$
with $i<j<p$ and $y(i) < y(j) < y(p)$. Let $j$ be the maximal integer with these properties;
then $y \lessdot y(j,p)$ and so $j \in L(z)$ by Theorem~\ref{taubruhat-thm}.
However, it cannot hold that $j \in E$ since this would contradict the fact that
$[y]_E\lessdot [y]_E(a,b)$, so  $L(z) \not\subset E$.
\end{proof}

We say that $z \in \I_\ZZ$ \emph{contains a bad $P$-pattern} if there exists a finite set $E\subset \ZZ$
which is $z$-invariant and which contains at most four $z$-orbits,
such that $[z]_E$ is not $P$-vexillary. In this situation we refer to the set $E$
as a \emph{bad $P$-pattern} for $z$.
We state two technical lemmas about this definition.

\begin{lemma}\label{vex-lem3}
If  $z \in \I_\ZZ$ is such that $|\iT_1(z)| \geq 2$, then $z$ contains a bad $P$-pattern.
\end{lemma}

\begin{proof}
Let $(q,r)$ be the maximal visible inversion of $z$, let $y = \ipsi(z)$, and let $p=y(q)\leq q$
so that $\iT_1(z) = \hat \Phi^-(y,p)$.
By hypothesis, there exist integers $i <j < p$ such that $\tau_{ip}(y)$ and $\tau_{jp}(y)$
are distinct elements of $\fk C(z) =\iT_1(z)$.
The set $E = \{ i,z(i), j, z(j), p,q,r, z(r) \}$ is  $z$-invariant and
 it holds by Lemma \ref{vex-lem2}(a) that $2 \leq | \fk C(z,E)| \leq |\fk C([z]_E)|$.
 Lemma \ref{vex-grass-lem} implies that $[z]_E$ is not I-Grassmannian,
 so $\iT_1([z]_E) = \fk C([z]_E)$
 and therefore $E$ is a bad $P$-pattern for $z$.
 \end{proof}

\begin{lemma}\label{vex-lem4}
Suppose  $z \in \I_\ZZ$ is such that $\iT_1(z) = \{v\}$ is a singleton set.
Then $z$ contains no bad $P$-patterns if and only if  $v$ contains no bad $P$-patterns.
\end{lemma}

\begin{proof}
It is a reasonable  computer calculation to check the following claim by brute force:
 if $z \in \I_{12}-\{1\}$ and $\fk C(z) = \{v\}$ is a singleton set, then 
$z$ contains no bad $P$-patterns if and only if  $v$ contains no bad $P$-patterns. (There are  73,843
such involutions $z$ to check.)

Now assume $z \in \I_\ZZ$ is such that $\iT_1(z) = \{v\}$ is a singleton set.
By construction, $v$ and $z$ have the same action on all integers outside 
a set $A\subset \ZZ$ of size at most 6. 
If $z$ (respectively, $v$) contains a bad $P$-pattern which is disjoint from $A$
then $v$ (respectively, $z$), clearly does as well.
If $z$ contains a bad $P$-pattern $B$ which is not disjoint from $A$,
then since $|B|\leq 8$ and since both $A$ and $B$ are $z$-invariant,
the set $E=A\cup B$ can have size at most $12$.
In this case, it follows from Lemma~\ref{vex-lem2}(b)
  that $\fk C([z]_E) = \{ [v]_E\}$
and that $[z]_E$ contains a bad $P$-pattern, so 
we deduce from the first paragraph that $[v]_E$  and therefore also $v$  contain bad $P$-patterns.
If instead $v$ contains a bad $P$-pattern disjoint from $A$,
then it follows by a similar argument that $z$ contains a bad $P$-pattern.
\end{proof}

We arrive at the main result of this section.

\begin{theorem}\label{vex-thm}
 An involution $z \in \I_\ZZ$
is $P$-vexillary if and only if $[z]_E$ is $P$-vexillary for all sets $E\subset \ZZ$ with $z(E) = E$
and $|E| = 8$.
\end{theorem}

\begin{proof}
Assume $z \in \I_\ZZ$ is not I-Grassmannian.
Corollary~\ref{it-cor1}(b) shows that $z$ is $P$-vexillary if and only if 
$\iT_1(z) = \{v\}$ and $v$ is $P$-vexillary.
Lemmas~\ref{vex-lem3} and \ref{vex-lem4} imply $z$ has no bad $P$-patterns if and only if
$\iT_1(z) = \{v\}$ and $v$ has no bad $P$-patterns. I-Grassmannian involutions have no bad $P$-patterns by Corollary~\ref{igrass-cor}.
Thus, by induction on the finite height of $\iT(z)$, an involution $z$ is $P$-vexillary if and only if it has no bad $P$-patterns,
which holds
if and only if 
$[z]_E$ is $P$-vexillary for all sets $E\subset \ZZ$ which are unions of at most four $z$-orbits. Since adding any number of sufficiently large fixed points of $z$ to $E$
will not change the symmetric function $\iF_{[z]_E}$, the last property holds if and only if it holds
for all sets $E\subset \ZZ$ with $E=z(E)$ and $|E|=8$.
%
%
\end{proof}

\begin{corollary}\label{ivex-cor}
An involution $z \in \I_\ZZ$ is $P$-vexillary if and only if for
all finite sets $E\subset \ZZ$ with $z(E) = E$
the standardization $[z]_E$ is not any of the following eleven permutations:
\[
\ba
(1, 2)(3 ,5), 	 &&&&& (1, 4)(3, 6), 		&&&&	(1 ,5)(2, 4)(3, 7),	&&&&	(1 ,6)(2 ,5)(3, 8)(4 ,7), \\
(1, 3)(4, 5), 	 &&&&& (1 ,4)(2 ,3)(5, 6), 	&&&&	(1, 5)(3, 7)(4, 6),	&&&&	(1, 6)(2, 4)(3 ,8)(5, 7), \\
			     &&&&& (1, 2)(3, 6)(4 ,5), 	&&&&					    &&&&	(1 ,3)(2, 5)(4 ,7)(6 ,8). \\
			     &&&&& (1, 2)(3 ,4)(5 ,6), 	&&&&
\ea
\]
\end{corollary}

\begin{proof}
Using Theorems \ref{igrass-thm} and \ref{it0-thm}, we have checked
by a computer calculation that  $z \in \I_8$ is not $P$-vexillary if and only if
there exists a $z$-invariant subset $E\subset \ZZ$ such that $[z]_E$
is one of the given involutions. The corollary therefore follows by Theorem \ref{vex-thm}.
\end{proof}

\begin{corollary}\label{ivex-cor2}
Suppose $z \in \I_\ZZ$ is 321-avoiding.
Then $z$ is $P$-vexillary if and only if 
for all finite $z$-invariant sets $E\subset \ZZ$,
it holds that $[z]_E$ is neither $(1,2)(3,4)(5,6)$ nor $(1,3)(2,5)(4,7)(6,8)$.
\end{corollary}

\begin{proof}
The other nine permutations in Corollary~\ref{ivex-cor} are not 321-avoiding, so the result follows.
\end{proof}

%
%
%
%
%
%
%

As an application, we give an alternate proof of a theorem of DeWitt \cite{DeWitt}.
A partition is a \emph{rectangle} if its nonzero parts are all equal.
The next statement is equivalent to \cite[Theorem V.3]{DeWitt}.

\begin{theorem}[DeWitt \cite{DeWitt}]\label{t:dewitt}
Fix a partition $\mu\subset \delta_m$. 
The skew Schur function
$s_{\delta_{m}\setminus \mu}$ is a Schur $P$-function
if and only if $\delta_m\setminus \mu$ is equivalent to $\delta_n\setminus \rho$
for a rectangle $\rho \subset \delta_n$ for some $n \in \PP$.
\end{theorem}

\begin{proof}
Let $\mu$ be a partition strictly contained in $\delta_{n+1}$
for some $n \in \NN$, and define $y=y_{\mu,n}$ as in \eqref{ymu-eq}.
By Proposition~\ref{ardila-prop1} and Lemma~\ref{ardila-lem2}, it suffices to show that 
 $y$ is $P$-vexillary if and only if $\mu$ is a rectangle.
If $\mu$ is a rectangle with $k$ parts of size $j$, then the numbers $b_i$ in \eqref{ymu-eq}
have the form $\{b_1,b_2,\dots,b_n\} = (n-k+ [j]) \cup ( n +j + [n-j] )$,
and it is an easy exercise to check that the 321-avoiding involution $y$ satisfies the conditions
in Corollary~\ref{ivex-cor2} so is $P$-vexillary.

Suppose that $\mu$ is not a rectangle.
Let $a_i$ and $b_i$ be as in \eqref{ymu-eq}
so that $a_1=1$ and $b_i = 2n$.
It is helpful to note that if $\mathcal{G}$ 
is the graph on  $[2n]$ with an edge from $i$ to $i+1$ for 
each $i \in [2n-1]$,
then  $\mu$ is not a rectangle if and only if
the induced subgraph of $\mathcal{G}$ on $\{a_1,a_2,\dots,a_n\}$
has at least three connected components.
Let $i \in [n]$ be maximal such that $a_i=i$ and let $j \in [n]$ be minimal such that
$b_j=n+j$.
If $i=1$, then $[y]_E = (1,2)(3,4)(5,6)$ for $E = \{a_1,b_1,a_2,b_2,a_n,b_n\}$.
If $j=n$ then 
$[y]_E = (1,2)(3,4)(5,6)$ for $E = \{a_1,b_1,a_{n-1},b_{n-1},a_n,b_n\}$.
If $i>1$ and $j<n$, then 
one checks that $[y]_E$ is  $(1,2)(3,4)(5,6)$ or $(1,3)(2,5)(4,7)(6,8)$ when $E$
is one of $\{ a_1,b_1,a_{i+1},b_{i+1},a_n,b_n\}$,
 $ \{ a_1,b_1,a_{j-1},b_{j-1},a_n,b_n\}$,
or $ \{a_1,b_1,a_{i+1},b_{i+1}, a_{j-1},b_{j-1},a_n,b_n\}$.
In either case, we conclude by Corollary~\ref{ivex-cor2} that $y$ is not $P$-vexillary, as required.
\end{proof}

\subsection{Schur $Q$-positivity}\label{schur-q-sect}

As in the introduction, define
$Q_\lambda = 2^{\ell(\lambda)} P_\lambda$ and $\iG_y = 2^{\kappa(y)} \iF_y$
for strict partitions $\lambda$ and involutions $y \in \I_\ZZ$,
where  $\kappa(y)$ is the number of nontrivial cycles of $y$.
One calls $Q_\lambda$ the \emph{Schur $Q$-function} of $\lambda$.
Our main results about the expansion of $\iF_y$ into Schur $P$-functions
may be rephrased as statements about the expansion of $\iG_y$ into Schur $Q$-functions.
We may restate Theorem~\ref{igrass-thm} as follows:

\begin{corollary}\label{q-grass-cor}
If $y \in \I_\ZZ$ is I-Grassmannian, then
$\iG_y = Q_{\ilambda(y)}$.
\end{corollary}

Recall the definition of $\iT_1(z)$ from Section~\ref{ppos-sect}
and $\fk C(z)$ from Section~\ref{i-vex-sect}.

\begin{lemma}
If $z \in \I_\ZZ$ and $v \in \iT_1(z)$ then $\kappa(v) \leq \kappa(z)$.
\end{lemma}

\begin{proof}
If $z \in \I_\ZZ$ is not I-Grassmannian and $y = \ipsi(z)$ and $v \in \iT_1(z) $,
then it is evident from
 Tables~\ref{ct-fig} and \ref{it-fig}
 that $\kappa(y) \leq \kappa(z)$ and  $\kappa(v) \leq \kappa(y)+1$
and  one of these inequalities must be strict.
\end{proof}

\begin{corollary}
If $z \in \I_\ZZ$ then
$\iG_z \in \NN\spanning\left\{ \iG_y : y \in \I_\ZZ\text{ is I-Grassmannian}\right\}$.
\end{corollary}

\begin{proof}
By Theorem~\ref{it0-thm}, $\iG_z = \sum_v 2^{\kappa(z)} \iF_v=\sum_v 2^{\kappa(z)-\kappa(v)} \iG_v$ where the sum is over the finite set of 
leaves in $\iT(z)$. The previous lemma implies that each coefficient is a positive integer.
\end{proof}

Combining the preceding statements gives this variant of Corollary~\ref{our-cor1}:

\begin{corollary}\label{schur-q-cor}
Each $\iG_y$ is Schur $Q$-positive, that is,
$\iG_y \in \NN\spanning \{Q_\lambda : \lambda\text{ is a strict partition}\}$.
\end{corollary}

Say that $z \in\I_\ZZ$ is \emph{$Q$-vexillary} if $\iG_z = Q_\lambda$ for a strict partition $\lambda$.
We can classify such permutations in much the same way as we did for $P$-vexillary involutions.

\begin{proposition}
If $z \in \I_\ZZ$ is $Q$-vexillary then $z$ is also $P$-vexillary.
\end{proposition}

\begin{proof}
If $\iG_z = Q_\lambda$ then $\iF_z = 2^e P_{\lambda}$ for some integer $e$, and
Corollary~\ref{ut-cor} implies that $e=0$.
\end{proof}

%

Recall the definition of a bad $P$-pattern from Section~\ref{i-vex-sect}.
Define a subset $E\subset \ZZ$ to be a \emph{bad $Q$-pattern}
for $z \in \I_\ZZ$ if $E$ is a union of at most four $z$-orbits and $[z]_E$ is not $Q$-vexillary.
The preceding proposition implies that
any bad $P$-pattern for $z \in \I_\ZZ$ is also a bad $Q$-pattern.

\begin{lemma}\label{q-vex-lem1}
Let $z \in \I_\ZZ$.
If either $|\iT_1(z)| \geq 2$ or  $\iT_1(z) = \{v\}$  where $\kappa(v) < \kappa(z)$,
then $z$ contains a bad $Q$-pattern.
\end{lemma}

\begin{proof}
In the first case, $z$ contains a bad $P$-pattern by Lemma~\ref{vex-lem3}.
In the second case, it follows from 
Lemmas~\ref{vex-grass-lem} and \ref{vex-lem2} that there exists a set $E\subset \ZZ$ composed of at most four $z$-orbits
such that $\iT_1([z]_E) = \fk C([z]_E) = \{ [v]_E\}$ and $\kappa(z) -\kappa(v) = \kappa([z]_E) - \kappa([v]_E) >0$.
This set $E$ is then
a bad $Q$-pattern for $z$ since
$\iG_{[z]_E} = 2^{\kappa(z)-\kappa(v)}\iG_{[v]_E}$ cannot be a Schur $Q$-function.
\end{proof}

\begin{lemma}\label{q-vex-lem2}
Suppose $z \in \I_\ZZ$
is such that $\iT_1(z) = \{v\}$ and $\kappa(v) = \kappa(z)$.
Then $z$ contains no bad $Q$-patterns if and only if $v$ contains no bad $Q$-patterns.
\end{lemma}

\begin{proof}
We have used a computer to check directly that if $z,v \in \I_{12}$
are 
such that 
$ \fk C(z)=\{v\}$
and $\kappa(v) = \kappa(z)$, then 
$z$ contains no bad $Q$-patterns if and only if $v$ contains no bad $Q$-patterns.
From this empirical fact, the result follows by the same 
argument as in the proof of Lemma~\ref{vex-lem4}.
\end{proof}

\begin{theorem}
An involution $z \in \I_\ZZ$ is $Q$-vexillary
if and only if $[z]_E$ is $Q$-vexillary
for all sets $E\subset \ZZ$ with $z(E)=E$ and $|E|=8$.
\end{theorem}

\begin{proof}
Assume $z \in \I_\ZZ$ is not I-Grassmannian.
It is clear from Corollary~\ref{it-cor1}(b) that $z$ is $Q$-vexillary
if and only if 
and  $\iT_1(z) = \{v\}$ where $v$ is $Q$-vexillary and $\kappa(v) = \kappa(z)$.
On the other hand, Lemmas~\ref{q-vex-lem1} and \ref{q-vex-lem2}
show that $z$ contains no bad $Q$-patterns if and only if
 $\iT_1(z) = \{v\}$ where $v$ contains no bad $Q$-patterns and $\kappa(v) = \kappa(z)$.
Since all I-Grassmannian involutions are $Q$-vexillary and contain no bad $Q$-patterns
by Corollaries~\ref{igrass-cor} and \ref{q-grass-cor},
the result  follows by induction on the height of the involution Lascoux-Sch\"utzenberger tree,
 as in the proof of Theorem~\ref{vex-thm}.
\end{proof}

The following is Theorem~\ref{q-vex-thm} 
in the introduction.

\begin{theorem}
An element of $\I_\ZZ$ is $Q$-vexillary if and only if it is vexillary, i.e., 2143-avoiding.
\end{theorem}

\begin{proof}[Proof of Theorem~\ref{q-vex-thm}]
We have checked by computer that $z \in \I_8$ is $Q$-vexillary if and only if 
for all finite sets $E\subset \ZZ$
with $z(E) = E$, the involution $[z]_E$ 
is
not
$
(1, 2)(3, 4)$ or $ (1, 4)(3, 6)$
or $  (1, 5)(3, 7)(4, 6)$ or $ (1, 5)(2, 4)(3, 7)$
or $  (1, 6)(2, 5)(3, 8)(4, 7)$.
The previous theorem implies that $z \in \I_\ZZ$ is $Q$-vexillary
 if and only if the same pattern avoidance condition holds.
If this condition fails then
 $z$ contains a 2143 pattern since none of the excluded involutions are vexillary.
 Conversely, suppose $z \in \I_\ZZ$ contains a 2143 pattern,
 so that $z(j) < z(i) < z(l) < z(k)$
 for integers $i<j<k<l$. Let    $E =\{i,j,k,l,z(i),z(j),z(k),z(l)\}$.
 One of the following must then occur:
 \begin{itemize}
\item There exists a set $F=z(F)\subset \ZZ$ with $[z]_F=(1,2)(3,4)$.
 
\item Among $i,j,k,l$ only $i$ or $j$ is a fixed point and $[z]_E = (1,5)(3,7)(4,6)$.

\item Among $i,j,k,l$ only $k$ or $l$ is a fixed point and $[z]_E = (1,5)(2,4)(3,7)$.

\item Exactly two of $i,j,k,l$ are fixed points  and  $[z]_E = (1,4)(3,6)$.

\item None of $i,j,k,l$ are fixed points and  $[z]_E = (1,6)(2,5)(3,8)(4,7)$.
       
 \end{itemize}
We conclude that if $z \in \I_\ZZ$ is not vexillary if and only if $z$ is not $Q$-vexillary.
\end{proof}

\subsection{Pfaffian formulas}\label{pfaffian-sect}

Let $y \in \I_\infty$ be I-Grassmannian.
In this section we prove a formula for $\iS_y$ inspired by a determinantal
expression for the Schur $P$-function $\iF_y = P_{\ilambda(y)}$.
Let  $\F_n$ be the set of fixed-point-free involutions in $S_n$.
The \emph{Pfaffian} of a skew-symmetric $n\times n$ matrix $A$
is the expression
\be\label{pf-eq}
 \pf A
= \sum_{z \in \F_n} (-1)^{\ellhat(z)+\frac{n}{2}} \prod_{z(i)<i \in [n] } A_{z(i),i}
.\ee
It is a classical fact that $\det A  = (\pf A)^2$. Since $\det A = 0$ when $A$ is skew-symmetric
but $n$ is odd, the definition \eqref{pf-eq}
is consistent with the fact that $\F_n$ is empty for $n$ odd.

\begin{example}\label{pf-ex}
If $A = (a_{ij})$ is a $2\times 2$ skew-symmetric matrix then $\pf A = a_{12} = -a_{21}$.
If $A = (a_{ij})$ is a $4\times 4$ skew-symmetric matrix then
$\pf A = a_{21} a_{43} - a_{31}a_{42} + a_{41}a_{32}$.
\end{example}

All matrices of interest in this section are skew-symmetric, and
we write $[a_{ij}]_{1\leq i<j\leq n}$ to denote the unique $n\times n$ skew-symmetric matrix with
$a_{ij}$ in entry $(i,j)$ for $i<j$ (and, necessarily, with $-a_{ij}$ in entry $(j,i)$, and 0
in each diagonal entry).
Observe that in this notation  $[1]_{1\leq i<j\leq n}$ is neither the identity matrix
nor the matrix whose entries are all 1's.

\begin{lemma} \label{pf-lem0}
Suppose $n \in \PP$ is even. Then
$\pf [1]_{1\leq i<j\leq n} =\sum_{z \in \F_n} (-1)^{\ellhat(z)+\frac{n}{2}}= 1$.
\end{lemma}

\begin{proof}
Let $\cX_n = \{ z \in \F_n : z(n-1) =n \}$ and $\cY_n = \F_n - \cX_n$.
Conjugation and multiplication by $s_{n-1}$ define bijections
$\cY_n\to \cY_n$ and $\F_{n-2} \to \cX_n$ reversing the sign of $(-1)^{\ellhat(z)}$.
Hence
$
\pf [1]_{1\leq i<j\leq n} = \sum_{z \in \F_n} (-1)^{\ellhat(z)+\frac{n}{2}}
= \sum_{z \in \cX_n} (-1)^{\ellhat(z)+\frac{n}{2}} = \pf [1]_{1\leq i<j\leq n-2}
$,
and the result follows by induction.
\end{proof}

Let $\phi = (\phi_1,\phi_2,\dots)$ be an integer sequence which has finitely many nonzero terms.
If $\phi$ is of finite length $r$,
then we identify $\phi$ with the infinite sequence with $\phi_i =0$ for all $i>r$.
Define 
\[\ell(\phi) = \max \{ i \in \PP : \phi_i \neq 0\}
\qquand
\ell^+(\phi)=
\begin{cases}
\ell(\phi)+1 &\text{if $\ell(\phi)$ is odd}
\\
\ell(\phi) &\text{otherwise}.
\end{cases}
\]
As a notational convenience we write $P_{\lambda_1\lambda_2\cdots\lambda_r}$
in place of $P_\lambda =P_{(\lambda_1,\lambda_2,\dots,\lambda_r)}$
for a strict partition $\lambda = (\lambda_1,\lambda_2,\dots,\lambda_r)$.
The following identity appears as \cite[Eq.\ (8.11), {\S}III.8]{Macdonald2}.

\begin{theorem}[Macdonald \cite{Macdonald2}]\label{jt-thm}
If $\lambda$ is a strict partition then
$P_\lambda = \pf[P_{\lambda_i\lambda_j}]_{1\leq i <j \leq \ell^+(\lambda)}$.
\end{theorem}

This theorem is an analogue of the Jacobi-Trudi identity for Schur functions,
which may be written succinctly as $s_\lambda = \det[s_{\lambda_i-i+j}]$.
The formula in Theorem \ref{jt-thm}  is what Schur gave as the original definition of $P_\lambda$
in \cite{Schur},
after specifying $P_{\lambda}$ for strict partitions $\lambda$ with $\ell(\lambda)\leq 2$.

\begin{example}
For $\lambda =(3,2,1)$, Theorem \ref{jt-thm} gives
$P_{\lambda} = P_{(3,2)} P_{(1)} - P_{(3,1)}P_{(2)} + P_{(2,1)}P_{(3)}$.
\end{example}

When $y \in \I_\infty$ is I-Grassmannian, Theorem \ref{jt-thm}
expresses $\iF_y$ as a Pfaffian in terms of  involution
Stanley symmetric functions of I-Grassmannian involutions with at most two nontrivial cycles.
We introduce some notation to make this idea more explicit.
Fix
\be
\label{nrphi-eq}
n,r \in \PP \qquand \phi \in \PP^r \text{ with }0 < \phi_1 < \phi_2<\dots <\phi_r\leq n.
\ee
We set $\phi_i=0$ for $i>r$.
Let
$y= (\phi_1, n+1)(\phi_2,n+2)\cdots (\phi_r,n+r) \in \I_\infty$
and define
\[
\iS[\phi_1,\phi_2,\dots,\phi_r;n]  = \fkS_y
\qquand
\iF[\phi_1,\phi_2,\dots,\phi_r;n]= \iF_y.
\]
When $r$ is odd, we also set $\iS[\phi_1,\phi_2,\dots,\phi_r,0;n] = \fkS_y$ and
$\iF[\phi_1,\phi_2,\dots,\phi_r,0;n]= \iF_y$.
Since $\iF_y = P_{ (n+1-\phi_1, \dots,n+1-\phi_r)}$
 by Theorem \ref{igrass-thm},
Theorem \ref{jt-thm} implies the following identity.

\begin{corollary}\label{jt-cor}
In the setup of \eqref{nrphi-eq},
 $\iF[\phi_1,\phi_2,\dots,\phi_r;n]
 = \pf\left[ \iF[\phi_i,\phi_j;n]\right]_{1\leq i < j\leq \ell^+(\phi)}.$
\end{corollary}

Our main result in this section is to show that the preceding formula
 is true even before stabilizing, that is, with $\iF[{\cdots};n]$ replaced by $\iS[{\cdots};n]$.
In the following lemmas, let
$\fk M[\phi;n]=\fk M[\phi_1,\phi_2,\dots,\phi_r;n]$ denote the skew-symmetric matrix
$\left[\iS[\phi_i,\phi_j;n]\right]_{1\leq i <j \leq \ell^+(\phi)}$.

\begin{lemma}\label{pf-lem1}
Maintain the notation of \eqref{nrphi-eq}, and suppose $p \in [n-1]$. Then
\[
\partial_p  \(\pf \fk M[\phi;n] \)=
\begin{cases}
\pf \fk M[\phi+e_i;n] &\text{if }p = \phi_i \notin \{ \phi_2-1, \dots,\phi_r-1\}\text{ for some $i \in [r]$} \\
0&\text{otherwise}
\end{cases}
\]
where $e_i = (0,\dots,0,1,0,0,\dots)$ is the standard basis vector
whose $i$th coordinate is 1.
\end{lemma}

\begin{proof}
Write $\fk M = \fk M[\phi;n]$. For indices $1\leq i < j \leq \ell^+(\phi)$, it follows from \eqref{i-eq} that
$\partial_p\fk M_{ij} =\partial_p \iS[\phi_i,\phi_j;n]$ is
$\iS[\phi_i +1, \phi_j]$ if $p = \phi_i  \neq \phi_j-1$,
$\iS[\phi_i, \phi_j+1]$ if $p=\phi_j$,
and 0 otherwise.
Thus, if $p \notin \{\phi_1,\phi_2,\dots,\phi_r\}$, then all entries of $\fk M$
are symmetric in $x_p$ and $x_{p+1}$,
so $\partial_p  \(\pf \fk M \) = 0$.
Assume $p = \phi_k$ for some $k \in [r]$.
Then $\partial_p\fk M_{ij}=0$ unless $i=k$ or $j=k$,
so it follows from
 \eqref{pf-eq} that
$\partial_p  \(\pf \fk M \) = \pf \fk N$ where $\fk N$
is the matrix formed by replacing the entries in the $k$th row and the $k$th column
of $\fk M$
by their images under $\partial_p$.
If $k<r$ and $\phi_k= \phi_{k+1}-1$, then columns $k$ and $k+1$ of $\fk N$ are identical,
so   $\pf \fk M= \pf \fk N = 0$ since $(\pf \fk N)^2 = \det \fk N = 0$.
If $k=r$ or if $k<r$ and $\phi_k \neq \phi_{k+1}-1$, then $\fk N = \fk M[\phi+e_k;n]$.
In either case the desired identity holds.
\end{proof}


\begin{lemma} \label{pf-lem2}
Let $n \in \PP$ and $D=x_1(x_1+x_2)(x_1+x_3)\cdots (x_1+x_n)$.
Then
$\pf \fk M[1;n]=D$,
and if $b\in \PP$ is such that $1<b\leq n$, then $\pf \fk M[1,b;n]$
is divisible by $D$.
\end{lemma}

Note that $\fk M[1;n]$ and $\fk M[1,b;n]$ are
both $2\times 2$ skew-symmetric matrices; cf. Example \ref{pf-ex}.

\begin{proof}
It follows from Theorem \ref{dominant-thm}
that $\pf \fk M[1;n] = \fkS_{(1,n+1)}=D$
and, when $n\geq 2$, that
 $\pf \fk M[1,2;n] = \fkS_{(1,n+1)(2,n+2)}=x_2(x_2+x_3)\cdots (x_2+x_n)D$.
 Assume $2 < b \leq n$
so that  $\pf \fk M[1,b;n]  =  \partial_{b-1} (\pf \fk M[1,b-1;n] )$
 by Lemma \ref{pf-lem1}.
By induction $\pf \fk M[1,b-1;n]  = qD$ for some polynomial $q$.
Since $D$ is symmetric in $x_{b-1}$ and $x_b$, we have
 $\pf \fk M[1,b;n] =   \partial_{b-1}(qD) = (  \partial_{b-1} q)D$ as desired.
\end{proof}

If $i : \PP\to \NN$ is a map with $i^{-1}(\PP)\subset [n]$ for some finite $n$,
then we define $x^i = x_1^{i(1)} x_2^{i(2)}\cdots x_n^{i(n)}$.
Given a nonzero polynomial
$f = \sum_{i : \PP \to \NN} c_i x^i $,
let $j : \PP \to \NN$
be the lexicographically minimal index
such that $c_j \neq 0$ and define $\minlex (f) = c_j x^j$.
We refer to $\minlex(f)$ as the
\emph{least term} of $f$.
Set $\minlex(0) = 0$,
so that $\minlex(fg) = \minlex(f) \minlex(g)$ for any polynomials $f,g$.
The following is \cite[Proposition 3.14]{HMP1}.

\begin{lemma}[See \cite{HMP1}]\label{minlex-lem}
If $y \in \I_\infty$ then
$\minlex(\iS_y) =  x^{\hat c(y)}= \prod_{(i,j) \in \hat D(y)} x_i $.
\end{lemma}

\begin{lemma}\label{pf-igrass-lem}
Let $i,j,n \in \PP$. The following identities then hold:
\ben
\item[(a)] If $i \leq n$ then $\minlex( \iS[i;n] ) = x_ix_{i+1}\cdots x_n$.
\item[(b)] If $i<j\leq n$ then $\minlex( \iS[i,j;n]) = (x_ix_{i+1}\cdots x_n)(x_jx_{j+1}\cdots x_n)$.
\een
\end{lemma}

\begin{proof}
If $i\leq n$ then $\iS[i;n] = \iS_y$ for $y=(i,n+1)$,
and if $i<j\leq n$ then $\iS[i,j;n] = \iS_z$ for $z=(i,n+1)(j,n+2)$.
One checks that
$\hat D(y) = \{ (i,i),(i+1,i),\dots,(n,i)\}$ and
$\hat D(z) = \{ (i,i), (i+1,i),\dots,(n,i)\} \cup \{ (j,j),(j+1,j),\dots,(n,j)\}$,
so the result follows by Lemma~\ref{minlex-lem}.
\end{proof}

The following proves the base case of this section's main result.

\begin{lemma}\label{pf-lem3}
If  $n \in \PP$ and $r \in [n]$ then 
$\iS[1,2,\dots,r;n] = \pf \fk M[1,2,\dots,r;n].$
\end{lemma}

\begin{proof}
Let $y = (1,n+1)(2,n+2)\cdots (r,n+r) \in \I_\infty$ and
$D_i = x_i(x_i+x_{i+1})(x_i+x_{i+2})\cdots(x_i+x_n)$ for $i \in [n]$, so that $D_n=x_n$.
As noted in the proof of Theorem \ref{igrass-lem4},
 Theorem \ref{dominant-thm} implies that
$\iS[1,2,\dots,r;n]  = \iS_y
= D_1D_2\cdots D_r$. 
Write $\fk M =  \fk M[1,2,\dots,r;n]$.
 Lemma \ref{pf-lem1} implies that $\partial_i (\pf \fk M) =0$ for each $i \in [r-1]$, so
$\pf \fk M$ is symmetric in the variables $x_1,x_2,\dots, x_r$.
By Lemma \ref{pf-lem2}, every entry in the first column of $\fk M$
is divisible by $D_1$, so $\pf \fk M$
is also divisible by $D_1$.
Since $s_i(D_i)$ is divisible by $D_{i+1}$ for $i \in [n-1]$ and since
$D_1,D_2,\dots,D_r$ are pairwise coprime,
we deduce that $\pf \fk M$ is divisible by $\iS[1,2,\dots,r;n]$.
Since both of these polynomials are homogeneous and one divides the other,
to prove they are equal it suffices to show that they have the same least term.

Let $m \in \PP$ be whichever of $r$ or $r+1$ is even
and choose $z \in \F_{m}$.
If $j \in [m]$ and $i=z(j)<j$, then $\fk M_{ij}$ is either
$\iS[i,j;n]$ (if $j<m$) or $\iS[i;n]$ (if $j=m$).
We compute by Lemma \ref{pf-igrass-lem} that
$
\minlex \(\prod_{z(i)<i \in [m] } \fk M_{z(i),i}\)= \prod_{z(i)<i \in [m] } \minlex \(\fk M_{z(i),i}\)
= (x_1x_2\cdots x_n)(x_2x_3\cdots x_n)\cdots (x_rx_{r+1}\cdots x_n)
$
which is precisely $\minlex( \iS[1,2,\dots,r;n])  = \minlex(D_1)\minlex(D_2)\cdots\minlex(D_r)$.
Since $\sum_{z \in \F_m} (-1)^{\ellhat(z) + \frac{m}{2}} = 1$ by Lemma \ref{pf-lem0},
we deduce that $\minlex( \pf \fk M) = \minlex( \iS[1,2,\dots,r;n])$ as needed.
\end{proof}


\begin{theorem}\label{pfaffian-thm}
In the setup of \eqref{nrphi-eq},
$\iS[\phi_1,\phi_2,\dots,\phi_r;n]
= \pf \left[\iS[\phi_i,\phi_j;n]\right]_{1\leq i <j \leq \ell^+(\phi)}.$
\end{theorem}

\begin{proof}
Writing $\iS[\phi;n]$ in place of $\iS[\phi_1,\phi_2,\dots,\phi_r;n]$,
we must show that $\iS[\phi;n] = \pf \fk M[\phi;n].$
As in the proof of Lemma \ref{igrass-lem4}, we proceed by induction on
$\Sigma(\phi) = \sum_{i=1}^r (\phi_i-i)$.
If $\Sigma(\phi) = 0$ then $\phi=(1,2,\dots,r)$
so $\iS[\phi;n] = \pf \fk M[\phi;n]$ by Lemma \ref{pf-lem3}.
Suppose instead that $\Sigma(\phi) >0$. Let $i \in [r]$ 
be the smallest index such that $i<\phi_i$ and set $p=\phi_i-1$.
Theorem \ref{ithm} then implies that 
$\iS[\phi;n] = \partial_p \iS[\phi-e_i;n]$,
while Lemma \ref{pf-lem1} implies 
that
$\pf \fk M[\phi;n] = \partial_p (\pf \fk M[\phi-e_i;n]).$
We may assume that $\iS[\phi-e_i;n] = \pf \fk M[\phi-e_i;n]$ by induction,
so $\iS[\phi;n] = \pf \fk M[\phi;n]$ as needed.
\end{proof}

\begin{example}
For $\phi = (1,2,3)$ and $n=3$ the theorem reduces to the identity
\[
\iS_{(1,4)(2,5)(3,6)}
=
\pf \(\barr{rrrr}
0 & \iS_{(1,4)(2,5)} & \iS_{(1,4)(3,5)} & \iS_{(1,4)} \\
-\iS_{(1,4)(2,5)} & 0 & \iS_{(2,4)(3,5)} & \iS_{(2,4)} \\
-\iS_{(1,4)(3,5)} & -\iS_{(2,4)(3,5)} & 0 & \iS_{(3,4)} \\
-\iS_{(1,4)} & -\iS_{(2,4)} & -\iS_{(3,4)} & 0
\earr\)
.\]
By Theorem \ref{dominant-thm}, both of these expressions evaluate to
$x_1x_2x_3(x_1+x_2)(x_1+x_3)(x_2+x_3)$.
\end{example}

It is an open question whether there exists
a simple, general formula for $\iS[i,j;n]$. 
If this were known, then
 the preceding result 
  would give an effective algorithm for computing any $\iS_y$. 

\section{Insertion algorithms}
\label{insertion-sect}

In this section we describe an insertion algorithm for involution words
in order to prove bijectively that  $\iF_y$ is Schur $P$-positive.
Conveniently, 
 the  algorithm we need turns out to be given by restricting the domain of 
a bijection already studied by Patrias and Pylyavskyy~\cite{PP} called  \emph{shifted Hecke insertion}.
Our goal here is more expository than in previous sections, and we have included a significant amount of background material in
Sections~\ref{ss:schur-p2} and \ref{ss:sh} in order to give a readable account of shifted Hecke insertion.
Our new results appear in Section~\ref{ss:inv-insert}.

\subsection{Shifted tableaux}\label{ss:schur-p2}

The \emph{diagram} of a partition $\lambda$ is the set $D_\lambda = \{(i,j) \in \PP\times \PP: j \leq \lambda_i\}.$
If $\lambda$ is a strict partition, then its \emph{shifted diagram} is the set
$D'_\lambda = \{(i, i + j - 1) : (i,j) \in D_\lambda\}.$
We orient the elements of $D_\lambda$ and $D'_\lambda$
in the same way as the positions in a matrix, and refer to the $i$th row or $j$th column of these sets according to this convention.
A \emph{tableau} (respectively, \emph{shifted tableau}) of shape $\lambda$ is a map  $D_\lambda \to \PP$
(respectively, $D'_\lambda \to \ZZ\setminus \{0\}$).
We write $T_{ij}$ for the image of $(i,j)$ under $T$, and refer to this number as the entry of $T$ in position $(i,j)$.

A shifted tableau $T$ is \emph{increasing} if its entries are positive and strictly increasing along each row and column.
Let $\prec $ be the total order on $\ZZ \setminus 0$ with $-1 \prec 1 \prec -2 \prec 2 \prec \dots.$
A shifted tableau $T$ is \emph{semi-standard}  if the following conditions hold:
\begin{itemize}
\item The entries of $T$ are weakly increasing with respect to $\prec$ along each row and column.
\item No two positions in the same column of $T$ contain the same positive number. 

\item No two positions in the same row of $T$ contain the same negative number.

\item Every entry of $T$ on the main diagonal $\{ (i,i) : i \in \PP\}$ is positive.
\end{itemize}
An (unshifted) tableau is defined to be \emph{semi-standard} in the same way,
but with the added constraint that its entries are all positive.
Finally, a (shifted) tableau $T$ of shape $\lambda$ is \emph{standard}
if it is semi-standard and $(i,j) \mapsto |T_{ij}|$ is a bijection  $D_\lambda \to [n]$ or $D'_\lambda \to [n]$ for some $n\in \NN$, as appropriate. 
Let $\SSYT(\lambda)$ and $\SYT(\lambda)$ be the sets of semi-standard and standard  tableaux of shape $\lambda$, respectively.
Similarly, when $\lambda$ is strict, let $\Inc(\lambda)$, $\SSMT(\lambda)$, and $\SMT(\lambda)$ be the sets of increasing, semi-standard, and standard shifted tableaux of shape $\lambda$.
Shifted tableaux as we have defined them are sometimes called \emph{shifted marked tableaux}---hence our use of the letters ``$\SMT$.''

\begin{example}\label{smt-ex}
Every semi-standard shifted tableau of shape $\lambda=(2,1)$ has the form
\[
\ytableausetup{boxsize = .5cm}
\begin{ytableau}
a &  b \\ 
\none & c
\end{ytableau}
\qquord
\begin{ytableau}
a &   -b \\ 
\none & c
\end{ytableau}
\qquord
\begin{ytableau}
a &   a \\ 
\none & b
\end{ytableau}
\qquord
\begin{ytableau}
a &   -b \\ 
\none & b
\end{ytableau}
\]
for some positive integers $a<b<c$.
The set $\SMT(\lambda)$ contains two elements, given by the first two tableaux shown here with $(a,b,c)=(1,2,3)$.
\end{example}

The following power series are the \emph{fundamental quasi-symmetric functions (of degree $n$)}:

\begin{definition}
\label{d:fund}
For $n \in \PP$ and $S\subset [n-1]$ define $f_{n,S} =\sum x_{i_1} x_{i_2} \dots x_{i_n}\in  \ZZ[[x_1,x_2,\dots]]$,
where the sum is over all weakly increasing sequences $(i_1, i_2, \dots, i_n) \in \PP^n$ with  $ i_j < i_{j+1} $ for all $ j \in S$.
\end{definition}

When $n$ is clear from context, we write $f_{S}$ in place of $f_{n,S}$.

\begin{remark}
If $\a = (s_{a_1},s_{a_2},\dots, s_{a_n})$ is a sequence of simple transpositions, then
the power series $f_\a$ defined in the introduction is
 $ f_{n,S}$ for $S = \{ i \in [n-1]: a_i < a_{i+1}\}$.
\end{remark}

For a (shifted) tableau $T$, define $x^T$ as the monomial given by the product $\prod_{(i,j)} x_{|T_{ij}|}$ over all $(i,j)$ in $T$'s domain.
When $T$ is standard, its \emph{descent set} if the set $\Des(T)$ of positive integers $i  $ such that 
either 
(1) $i$ and $i+1$ both appear in $T$ with $i$ in a row strictly above $i+1$,
(2) $-i$ and $-(i+1)$ both appear in $T$ with $-i$ in a column strictly to the left of $-(i+1)$,
or 
(3) $i$ and $-(i+1)$ both appear in $T$.
If $T \in \SYT(\lambda)$ then only condition (1) can occur, and if $n=|\lambda|$ then
  $\Des(T)$ is the complement in $[n-1]$ of the descent set of the transpose of $T$, which 
 is also a standard tableau.
The following identities 
are well-known.

\begin{proposition}
\label{p:schur-fund}
If $\lambda$ is a partition of $n$
then
$s_\lambda = \sum_{T \in \SSYT(\lambda)} x^T = \sum_{T \in \SYT(\lambda)} f_{\Des(T)}
$.
\end{proposition}


\begin{proposition}
\label{p:schur-p-def}
If $\lambda$ is a strict partition of $n$
then
$
P_\lambda = \sum_{T \in \SSMT(\lambda)} x^T = \sum_{T \in \SMT(\lambda)} f_{\Des(T)} = \sum_{T \in \SMT(\lambda)} f_{[n-1]- \Des(T)}
.$
\end{proposition}

\begin{proof}
We include a proof for completeness.
The first equality is \cite[Eq.\ (8.16$'$), {\S}III.8]{Macdonald2} or \cite[Eq.\ (6.4)]{Stem},
and the second follows as an exercise.
Since the power series $\{f_\a\}$ are linearly independent
and since $P_\lambda$ is invariant under the linear automorphism of $\Lambda$
with $s_\mu \mapsto s_{\mu^T}$  \cite[Example 3(a), {\S}III.8]{Macdonald2},
the third equality follows upon noting that
$s_{\mu^T} = \sum_{T \in \SYT(\mu)} f_{[|\mu|-1]-\Des(T)}$.
\end{proof}

\begin{example} If $\lambda=(1,1)$ then $\SYT(\lambda)$ contains a single element whose descent set is $\{1\}$, so $s_{(1,1)} = \sum_{i<j} x_ix_j$. If $\lambda=(2,1)$ then Example~\ref{smt-ex} shows that $\SMT(\lambda)$ has two elements, whose descent sets are $\{2\}$ and $\{1\}$, so $P_{(2,1)} = 2\sum_{i<j<k} x_ix_jx_k + \sum_{i<j} x_i^2x_j + \sum_{i<j} x_ix_j^2$.
\end{example}

In Section~\ref{insertion-sect}, we will need one other family.
Fix a strict partition $\lambda$.
A \emph{set-valued shifted tableau} $T$ of shape $\lambda$ is a map from  $D_\lambda'$ to the set of finite, nonempty subsets of $\ZZ\setminus \{0\}$. 
A set-valued shifted tableau $T$ of shape $\lambda$ is
\emph{increasing} if each shifted tableau $U$ of shape $\lambda$ with $U_{ij} \in T_{ij}$ for all $(i,j)$ is increasing,
and \emph{standard (of rank $n$)} if each shifted tableau $U$ of shape $\lambda$ with $U_{ij} \in T_{ij}$ for all $(i,j)$ is semi-standard and 
the map $x \mapsto |x|$ is a bijection $\bigsqcup_{(i,j) \in D'_{\lambda}} T_{ij} \to  [n]$,
where $\bigsqcup$ denotes disjoint union. Let $\SetMT_n(\lambda)$ be the set of standard set-valued shifted tableaux of rank $n$.
Any shifted tableau may be viewed as a set-valued shifted tableau whose entries are all singleton sets,
and with respect to this identification it holds that $\SMT(\lambda) = \SetMT_n(\lambda)$ for $n=|\lambda|$.

\subsection{Shifted Hecke insertion}
\label{ss:sh}

We now present the definition of shifted Hecke insertion from \cite{PP}.
The simplest implementation requires three methods: a bumping rule, an insertion rule, and a final algorithm.
In what follows, we write $:=$ to denote the assignment of an expression on the right to a variable on the left.

\begin{algorithm}[Bumping rule] This algorithm takes a number,  a binary digit, and a possibly empty increasing sequence as inputs, and outputs three values of the same types.
\ben
\item[] Inputs: $p \in \PP$, $\Schensted \in \{0,1\}$, and $M = (m_1 < m_2 < \dots < m_n) \in \PP^n$.
\item[] Pseudo-code:
\ben
\item[B1:] If $p > m_n$ then set $M' := (m_1,m_2, \dots, m_n ,p)$ and $q:=0$.
\item[B2:] Else if $p = m_n$ then set $M' := M$ and $q:=0$.
\item[B3:] Else if $p < m_n$: 
\ben
\item[B4:] Let $i \in [n]$ be such that $m_{i-1} < p \leq m_i $ where $m_0 := -\infty$.
\item[B5:] If $i= 1$ then set $\Schensted:=1$.
\item[B6:] If $p=m_i $ then set $M' := M$ and $q := m_{i+1}$.
\item[B7:] Else if $p<m_i$ then set $M' := (m_1, \dots, m_{i-1}, p, m_{i+1}, \dots, m_n)$ and $q := m_i$.
\een
\item[B8:] Return $(q, \Schensted, M')$.
\een
\een
We denote the output of this algorithm as $\Bump(p,\Schensted,M)$.
\end{algorithm}

Informally, when inserting $p$ into $M$, we bump the first entry larger than $p$ and replace it with $p$ if the resulting sequence is increasing.
Here, $q$ is the entry bumped while $\Schensted$ is used in the insertion rule.

\begin{algorithm}[Insertion rule] This algorithm takes a number and an increasing shifted tableau as inputs, and outputs an index, a binary digit, and an increasing shifted tableau.
\ben
\item[]  Inputs: $p \in \PP$ and $P$ an increasing shifted tableau.

\item[] Pseudocode:
\ben
\item[I1:] Set $j := 0$ and $\Schensted:=0$.
\item[I2:] While $p > 0$:
\ben
\item[I3:] Set $j := j + 1$, and define $R$ and $C$ as the $j$th row and column of $P$.
\item[I4:] If $\Schensted = 0$: 
\ben
\item[I5:] Let $(p, \Schensted, R' ) := \Bump(p,\Schensted,R)$.
\item[I6:] Let $P'$ be the shifted tableau given by replacing the $j$th row of $P$ by $R'$.
\een
\item[I7:] Else if $\Schensted = 1$:
\ben
\item[I8:] Let $(p, \Schensted, C' ) := \Bump(p,\Schensted,C)$.
\item[I9:] Let $P'$ be the shifted tableau given by replacing the $j$th column of $P$ by $C'$.
\een
\item[I10:] If $P'$ is increasing then set $P := P'$.
\een
\item[I11:] Return  $(j, \Schensted, P)$.
\een
\een
We denote the output of this algorithm as  $\Insert(p,P)$.
\end{algorithm}

We apply the bumping rule row by row until there is no output ($q = 0$) or $\Schensted = 1$, at which point we begin bumping column by column until there is no output.
Combining the bumping rule and insertion rule as follows gives \emph{shifted Hecke insertion}.
(As its named suggests, this insertion algorithm  is
 a shifted analogue of \emph{Hecke insertion} as defined in \cite{buch2008stable}.)


\begin{algorithm}[Shifted Hecke insertion \cite{PP}] This algorithm takes a word as input and outputs an increasing shifted tableau and a standard set-valued shifted tableau, both of the same shape.
\ben
\item[] Inputs: $\a = (a_1,a_1, \dots, a_n) \in \PP^n$.

\item[] Pseudocode:
\ben
\item[S1:] Set $P := \varnothing$, $Q := \varnothing$, $\lambda := \varnothing$, and $\Schensted := 0$.

\item[S2:] For $i :=1,2,\dots,n$:
\ben
\item[S3:] Set $(j, \Schensted, P):= \Insert(a_i,P)$ and $\lambda := \mathrm{shape}(P)$, and let $R$ and $C$ denote the $j$th row and column of $P$.
\item[S4:] If $\Schensted= 0$, add $i$ to the last position in column $|R|$ of $Q$, so that $\mathrm{shape}(Q) = \lambda$.
\item[S5:] If $\Schensted =1$, add $-i$ to the last position in row $|C|$ of $Q$,  so that $\mathrm{shape}(Q) = \lambda$.
\een
\item[S6:] Return $(P,Q) \in \Inc(\lambda) \times  \SetMT_n(\lambda)$.
\een
\een
We denote the output of this algorithm as  $\SH(\a) = (P_{\SH}(\a),Q_{\SH}(\a)) $,
and refer to $P_{\SH}(\a)$ as the \emph{insertion tableau} and
$Q_{\SH}(\a)$ as the \emph{recording tableau}.
\end{algorithm}

\begin{remark}
In most insertion algorithms, each new entry in the recording tableau goes in the same position as the 
entry just inserted into the insertion tableau.
However, with set-valued recording tableaux, this position need not be a corner.
Steps S4 and S5 resolve this by translating the position of the recording tableau's new entry  to the bottom of its column (row insertion) or end of its row (column insertion).
See Figure~\ref{f:sh} for an example.
\end{remark}

\begin{figure}[h]
\caption{
We compute $\SH(\a) = (P_{\SH}(\a), Q_{\SH}(\a))$ for $\textbf{a} = (5,4,1,3,4,5,2,1,2)$.
For convenience, we write $i'$ in place of $-i$ in all shifted tableaux,
and define $\a[i] = (a_1,a_2, \dots, a_i)$.
}\label{f:sh}
\[
\ba
\ytableausetup{mathmode, boxsize=1.5em}
&P_{\SH}(\textbf{a}[1]) = 
\begin{ytableau}
5
\end{ytableau}
\qquad
&&Q_{\SH}(\textbf{a}[1]) = 
\begin{ytableau}
1
\end{ytableau}
\\
&P_{\SH}(\textbf{a}[2]) = 
\begin{ytableau}
4 & 5
\end{ytableau}
\qquad
&&Q_{\SH}(\textbf{a}[2]) = 
\begin{ytableau}
1 & 2'
\end{ytableau}
\\
&P_{\SH}(\textbf{a}[3]) = 
\begin{ytableau}
1 & 4 & 5
\end{ytableau}
\qquad
&&Q_{\SH}(\textbf{a}[3]) = 
\begin{ytableau}
1 & 2' & 3'
\end{ytableau}
\\
&P_{\SH}(\textbf{a}[4]) = 
\begin{ytableau}
1 & 3 & 5\\
\none & 4
\end{ytableau}
\qquad
&&Q_{\SH}(\textbf{a}[4]) = 
\begin{ytableau}
1 & 2' & 3'\\
\none & 4
\end{ytableau}
\\
&P_{\SH}(\textbf{a}[5]) = 
\begin{ytableau}
1 & 3 & 4\\
\none & 4 & 5
\end{ytableau}
\qquad
&&Q_{\SH}(\textbf{a}[5]) = 
\begin{ytableau}
1 & 2' & 3'\\
\none & 4 & 5
\end{ytableau}
\\
&P_{\SH}(\textbf{a}[6]) = 
\begin{ytableau}
1 & 3 & 4 & 5\\
\none & 4 & 5
\end{ytableau}
\qquad
&&Q_{\SH}(\textbf{a}[6]) = 
\begin{ytableau}
1 & 2' & 3' & 6\\
\none & 4 & 5
\end{ytableau}
\\
&P_{\SH}(\textbf{a}[7]) = 
\begin{ytableau}
1 & 2 & 4 & 5\\
\none & 3 & 5
\end{ytableau}
\qquad
&&Q_{\SH}(\textbf{a}[7]) = 
\begin{ytableau}
1 & 2' & 3' & 67'\\
\none & 4 & 5
\end{ytableau}
\\
&P_{\SH}(\textbf{a}[8]) = 
\begin{ytableau}
1 & 2 & 3 & 4 & 5\\
\none & 3 & 5
\end{ytableau}
\qquad
&&Q_{\SH}(\textbf{a}[8]) = 
\begin{ytableau}
1 & 2' & 3' & 67' & 8'\\
\none & 4 & 5
\end{ytableau}
\\
&P_{\SH}(\textbf{a}[9]) = 
\begin{ytableau}
1 & 2 & 3 & 4 & 5\\
\none & 3 & 5
\end{ytableau}
\qquad
&&Q_{\SH}(\textbf{a}[9]) = 
\begin{ytableau}
1 & 2' & 3' & 67' & 8'\\
\none & 4 & 59'
\end{ytableau}
\ea
\]
\end{figure}

The following key property is \cite[Theorem 5.18]{PP}.
\begin{theorem}[Patrias and Pylyavskyy \cite{PP}]
\label{t:sh-bij}
For all $n \in \NN$,
shifted Hecke insertion is a bijection  
$ \SH : \PP^n \to \bigcup_{\lambda\text{ strict}} \Inc(\lambda) \times \SetMT_n(\lambda).$
\end{theorem}

Define a \emph{word} to be a finite sequence of positive integers.
The \emph{descent set} of a word $\a = (a_1, a_2,\dots, a_n)$ is 
$\Des(\a) = \{i \in [n-1]: a_i > a_{i+1}\}$.
Let $T \in \SetMT_n(\lambda)$ be a standard set-valued shifted tableau, and say that
$x$ appears in position $(i,j)$ of  $T$ if $x \in T_{ij}$. 
Following \cite[Section 3.2]{hamaker2015shifted}, we define the \emph{descent set} of $T$
as the set $\Des(T)$ of positive integers $i$ 
which satisfy one of the following mutually exclusive conditions:
(1) $i$ and $i+1$ both appear in $T$ and $i$ is in a row strictly above $i+1$,
(2) $i$ and $-(i+1)$ both appear in $T$,
(3) $-i$ and $-(i+1)$ both appear in $T$ and $-(i+1)$ is in a row strictly above $-i$, or
(4) $-i$ and $-(i+1)$  appear in the same row of $T$ but not  in the same position.
One can check that if every entry of $T$ is a singleton set,
then this definition reduces to the descent set of a standard shifted tableau defined in 
Section~\ref{ss:schur-p2}.
We recall three properties of shifted Hecke insertion from \cite{hamaker2015shifted}.
The first is noted before \cite[Theorem 3.7]{hamaker2015shifted}:

\begin{proposition}[see~\cite{hamaker2015shifted}]
\label{p:sh-des}
For any word $\a$, we have $\Des(\a) = \Des(Q_{\SH}(\a))$.
\end{proposition}

The following equivalence relation was introduced in~\cite{buch2014k}.

\begin{definition}\label{d:k-knuth}
The \emph{weak $K$-Knuth moves} are the relations on words given by
\begin{center}
\begin{enumerate}
\item $(a,b, \dots) \ksim (b,a, \dots)$
\item $(\dots, a,c,b, \dots) \ksim (\dots ,c,a,b, \dots)$
\item $(\dots ,b,a,c, \dots) \ksim (\dots ,b,c,a, \dots)$
\item $(\dots ,a,b,a, \dots) \ksim (\dots ,b,a,b, \dots)$
\item $(\dots ,a,a, \dots)  \ksim (\dots ,a, \dots)$
\end{enumerate}
\end{center}
for any integers $a < b < c$, where in these expressions corresponding ellipses denote matching subsequences.
Two words $\a$ and $\b$ are \emph{weak $K$-Knuth equivalent}, denoted $\a \ksim \b$, 
if there exists a sequence of weak $K$-Knuth moves transforming $\a$ to $\b$.
\end{definition}

The following statement is \cite[Corollary 2.18]{hamaker2015shifted}; its converse does not hold.

\begin{proposition}[see~\cite{hamaker2015shifted}]
\label{p:k-knuth}
Let $\a,\b$ be words such that $P_{\SH}(\a) = P_{\SH}(\b)$.
Then $\a \ksim \b$.
\end{proposition}

For an integer-valued tableau $T$, define  $\rho(T)$ as  the sequence $(a_1,a_2,\dots,a_k)$ given by reading the rows of $T$ from bottom to top, reading the entries in each row from left to right.
For example, the reading word of $P_{\SH}(\a)$ in Figure~\ref{f:sh}
is $(3,5,1,2,3,4,5)$.
The following is implicit in \cite{hamaker2015shifted}.

\begin{proposition}[See \cite{hamaker2015shifted}]
\label{p:reading}
If $\lambda$ is a strict partition and $P \in \Inc(\lambda)$, then $P_{\SH}(\rho(P)) = P$.

\end{proposition}


\begin{lemma}
\label{l:knuth-inv}
Let $\a = (a_1, a_2,\dots ,a_k)$ and $\b = (b_1,b_2, \dots, b_l)$ be words with  $\a \ksim \b$.
If $v,w \in S_\infty$  are given by $v = s_{a_1} \circ  s_{a_2}\circ \dots \circ s_{a_k}$ and $w = s_{b_1} \circ s_{b_2}\circ \dots \circ s_{b_l}$, then $v^{-1} \circ v = w^{-1} \circ w$.

\end{lemma}
\begin{proof}
We may assume that $\a$ and $\b$ differ by a single weak $K$-Knuth move. 
If this move is (2)-(5) in Definition~\ref{d:k-knuth} then $v = w$. In the remaining case, one can check directly that $v^{-1} \circ v = w^{-1} \circ w$.
\end{proof}

\subsection{Involution Coxeter-Knuth insertion}
\label{ss:inv-insert}

Recall the definition of $\iR(y)$ from the introduction.
Restricting shifted Hecke insertion to this set 
gives both a shifted variant of Edelman-Greene insertion~\cite{EG} and a ``reduced word'' variant of Sagan-Worley insertion~\cite{sagan1987shifted,worley1984theory}. 
To refer to this map, we introduce the following terminology. 
\begin{definition}
For $y \in \I_\infty$ with $n=\ellhat(y)$,   \emph{involution Coxeter-Knuth insertion} is
the map
\be\label{ick}
 \iR(y) \longrightarrow \PP^n \xrightarrow{\SH}  \bigcup_{\lambda\text{ strict}} \Inc(\lambda) \times \SetMT_n(\lambda)
 \ee
where the first arrow is the inclusion $(s_{a_1},s_{a_2},\dots,s_{a_n}) \mapsto (a_1,a_2,\dots,a_n)$.
\end{definition}

 With slight abuse of notation, we denote the map \eqref{ick} also by $\SH$.
 For $\a \in \iR(y)$, define $\hat P(\a)$  and $\hat Q(\a)$ as the increasing/set-valued shifted tableaux
 such that $(\hat P(\a), \hat Q(\a)) = \SH(\a)$.
In the following results, just to make our notation consistent,
 redefine $\rho(T)$ for an integer-valued tableau $T$ to be the sequence $(s_{a_1},s_{a_2},\dots,s_{a_n})$
 where $(a_1,a_2,\dots,a_n)$ is the usual reading word of $T$.
 
\begin{lemma}\label{l:t-insertion}
Let $y \in \I_{\infty}$ and $\a \in \iR(y)$. Then $ \rho(\hat P(\a)) \in \iR(y)$ and  $\hat Q(\a) \in \bigcup_{\lambda\text{ strict}} \SMT(\lambda).$
\end{lemma}

\begin{proof}
Note that $\hat Q(\a)$ is a standard set-valued shifted tableau by  Theorem~\ref{t:sh-bij}, and so belongs to $\SMT(\lambda)$ for some $\lambda$
if and only if all of its entries are singleton sets.
Write $\a = (s_{a_1},s_{a_2},\dots, s_{a_n})$ and
define $\b = (s_{b_1},s_{b_2}, \dots, s_{b_m}) = \rho(\hat P(\a))$. By definition
 $n \geq m$, and it holds that all entries of $\hat Q(\a)$ are singletons if and only if $n=m$.
Let
$v=s_{a_1}\circ s_{a_2}\circ \dots \circ s_{a_n}$ and $w = s_{b_1} \circ s_{b_2}\circ \dots \circ s_{b_m}$.
Propositions~\ref{p:k-knuth} and \ref{p:reading} 
imply that $(a_1,a_2,\dots,a_n) \ksim (b_1,b_2,\dots,b_m)$,
so 
 $y = v^{-1}\circ v = w^{-1} \circ w$ by Lemma~\ref{l:knuth-inv}.
 Since $v \in \cA(y)$, this implies that $n=m$, so $\b \in \iR(y)$ as desired.
 \end{proof}

Putting together all of the preceding facts,
we arrive at the main result of this section.

\begin{theorem}
\label{t:insertion}
Let $y \in \I_{\infty}$.
Then involution Coxeter-Knuth insertion is a bijection
$
\iR(y) \to
\bigcup_{\lambda} \left\{(P,Q) \in \Inc(\lambda) \times \SMT(\lambda): \rho(P) \in \iR(y)\right\}
$
where the union is  over strict partitions $\lambda$ of $\ellhat(y)$.
\end{theorem}


\begin{proof}
The given map is a well-defined injection by Theorem~\ref{t:sh-bij}
and Lemma~\ref{l:t-insertion}.
To see that it is surjective, let $\lambda$ be a strict partition
and suppose $(P,Q) \in \Inc(\lambda) \times \SMT(\lambda)$ is such that $\rho(P) \in \iR(y)$,
so that $|\lambda| = \ellhat(y)$.
Since $\SH$ is a bijection,
 there exists a unique word $\a$ with $P_{\SH}(\a) = P$ and $Q_{\SH}(\a) = Q$.
No entry of $Q$ contains multiple values, so
  the length of $\a$ must also be $|\lambda|=\ellhat(y)$.
 By Propositions~\ref{p:k-knuth} and \ref{p:reading} and Lemma~\ref{l:knuth-inv},  replacing the entries of $\a$ by
  simple transpositions therefore gives an element of $\iR(y)$ whose image under $\SH$ is $(P,Q)$.
\end{proof}

\begin{remark}
These results show that involution Coxeter-Knuth insertion may be defined by a 
slightly simpler procedure than $\SH$.
Since $\hat Q(\a) \in \bigcup_\lambda \SMT(\lambda)$ for $\a \in \iR(y)$,
when computing involution Coxeter-Knuth insertion the following holds: 
(1) step B2 is superfluous in the Bumping rule, 
(2) in step I10 of the Insertion rule, $P'$ is always increasing,
and (3) in Steps S4/5 of shifted Hecke insertion, the last position in column/row $j$ is also the last position
in its respective row/column.
\end{remark}

\begin{example}
\label{e:sck}
For the involution word $\a = (s_3,s_5,s_4,s_1,s_2,s_3) \in \cR(246135)\subset \iR(456123)$,
 the sequence of tableaux obtained by involution Coxeter-Knuth insertion is as follows:
\[
\barr{lllllllllll}
\ytableausetup{boxsize = .5cm}
\begin{ytableau}
3
\end{ytableau}
&\to&
\begin{ytableau}
3 & 5
\end{ytableau} 
&\to&
\begin{ytableau}
3 & 4\\
\none & 5
\end{ytableau}
&\to&
\begin{ytableau}
1 & 3 & 4\\
\none & 5
\end{ytableau}
&\to&
\begin{ytableau}
1 & 2 & 4\\
\none & 3 & 5
\end{ytableau}
&\to&
\begin{ytableau}
1 & 2 & 3\\
\none & 3 & 4\\
\none & \none & 5
\end{ytableau} = P_{\SH}(\a)
\\
\\
\begin{ytableau}
1
\end{ytableau}
&\to&
\begin{ytableau}
1 & 2
\end{ytableau} 
&\to&
\begin{ytableau}
1 & 2\\
\none & 3
\end{ytableau}
&\to&
\begin{ytableau}
1 & 2 & 4'\\
\none & 3
\end{ytableau}
&\to&
\begin{ytableau}
1 & 2& 4'\\
\none & 3 & 5'
\end{ytableau}
&\to&
\begin{ytableau}
1 & 2 & 4'\\
\none & 3 & 5'\\
\none & \none & 6
\end{ytableau} = Q_{\SH}(\a).
\earr
\]
\end{example}


\begin{corollary}\label{t:insertion-cor}
Let $y \in \I_{\infty}$.
Then $\iF_y = \sum_\lambda  \beta_{y,\lambda} P_\lambda$ where the sum is over all strict partitions $\lambda$ and
$ \beta_{y,\lambda}$ is the number of increasing shifted tableaux $P$ of shape $\lambda$  with $\rho(P) \in \iR(y)$.
\end{corollary}

\begin{proof}
Since   $\iF_y = \sum_{\a \in \iR(y)}f_\a = \sum_{\a \in \iR(y)} f_{[n-1]-\Des(\a)}$ for $n=\ellhat(y)$,
where  $\Des(\a)$ is defined in the usual way,
the result is immediate from 
Propositions~\ref{p:schur-p-def} and \ref{p:sh-des} and Theorem~\ref{t:insertion}.
\end{proof}

\begin{corollary}
Involution Coxeter-Knuth insertion is a bijection $\iR(w_n) \to \SMT(\hat\delta_n)$ for each $n \in \NN$,
where $\hat\delta_n$ denotes the
strict partition $(n-1,n-3,n-5,\dots)$.
\end{corollary}

\begin{proof}
This follows from Theorem~\ref{t:insertion} and Corollary~\ref{t:insertion-cor} since $\iF_{w_n} = P_{(n-1,n-3,n-5,\dots)}$.
\end{proof}

%
%
%

Define \emph{shifted Coxeter-Knuth equivalence} to be the equivalence relation generated by (1)-(4) in Definition~\ref{d:k-knuth}.
We conjecture this analogue of both \cite[Theorem 6.24]{EG} and \cite[Theorem~7.2]{sagan1987shifted}.

\begin{conjecture}
\label{c:ick}
Two involution words $\a$, $\b$  are shifted Coxeter-Knuth equivalent if and only if 
they have the same insertion tableau under the map $\SH$, that is, $P(\a) = P(\b)$.
\end{conjecture}

If this conjecture were true, then one would be able to apply the approach outlined in~\cite{hamaker2014relating} to relate shifted Coxeter-Knuth insertion to involution Little bumps, as defined in~\cite{HMP3}.

\end{document}